\documentclass[12pt, letterpaper,reqno]{amsart}
\usepackage{amsthm,amssymb,mathrsfs}
\usepackage{amsmath,amscd}
\usepackage{mathtools}
\usepackage{xcolor}
\usepackage{graphicx} 
\usepackage{tikz}
\usepackage{pinlabel}
\usepackage{fourier}
\usepackage[breaklinks,colorlinks,citecolor=blue,linkcolor=blue,
 urlcolor=teal]{hyperref} 
\definecolor{mygreen}{HTML}{278444}
\theoremstyle{plain}
\newtheorem{theorem}[equation]{Theorem}
\newtheorem{proposition}[equation]{Proposition}
\newtheorem{lemma}[equation]{Lemma}
\newtheorem{corollary}[equation]{Corollary}

\theoremstyle{definition}
\newtheorem{definition}[equation]{Definition}
\newtheorem{remark}[equation]{Remark}

\newtheorem{example}[equation]{Example}
\newtheorem{notation}[equation]{Notation}
\newtheorem{question}[equation]{Question}
\newtheorem*{ack}{Acknowledgments}

\newcommand{\co}{\colon \thinspace}
\newcommand{\Q}{{\mathbb Q}}
\newcommand{\R}{{\mathbb R}}
\newcommand{\Z}{{\mathbb Z}}

\newcommand{\Hyp}{{\mathbb H}}

\newcommand{\abs}[1]{\lvert {#1} \rvert}

\renewcommand{\leq}{\leqslant}
\renewcommand{\geq}{\geqslant}
\renewcommand{\epsilon}{\varepsilon}
\renewcommand{\phi}{\varphi}
\newcommand{\N}{{\mathbb N}}

\DeclareMathOperator{\id}{id}

\DeclareMathOperator{\Cay}{Cay}
\DeclareMathOperator{\Aut}{Aut}
\DeclareMathOperator{\divideby}{div}
\DeclareMathOperator{\lcm}{lcm}

\DeclareMathOperator{\image}{Im}
\DeclareMathOperator{\Vol}{Vol}
\DeclareMathOperator{\cells}{cells}
\DeclareMathOperator{\Out}{Out}

\newenvironment{pict}[2]%
	{\setlength{\unitlength}{1mm}
	\begin{center}
	\begin{picture}(#1,#2)
	\scriptsize
}%
	{\end{picture}
	\end{center}

	\noindent}
\newcommand{\zindex}[3]{\put(#1,#2){\makebox(0,0){${#3}$}}}

\numberwithin{equation}{section}

\makeatletter
  \def\tagform@#1{\maketag@@@{%
   \textbf{(\ignorespaces#1\unskip\@@italiccorr)}}}%
   \renewcommand{\eqref}[1]{\textup{\maketag@@@{(\ignorespaces%
        {\ref{#1}}\unskip\@@italiccorr)}}}
\makeatother

\begin{document}

\title[Incommensurable lattices]{Incommensurable lattices in
  Baumslag--Solitar complexes}

\author{Max Forester}
\address{Mathematics Department\\
        University of Oklahoma\\
        Norman, OK 73019\\
        USA}

\email{mf@ou.edu}

\begin{abstract}
This paper concerns locally finite $2$--complexes $X_{m,n}$ which are
combinatorial models for the Baumslag--Solitar groups $BS(m,n)$. We
show that, in many cases, the locally compact group $\Aut(X_{m,n})$
contains incommensurable uniform lattices. The lattices we construct
also admit isomorphic Cayley graphs and are finitely presented,
torsion-free, and coherent. 
\end{abstract}

%% MR NUMBERS: 20E08, 20F65, 57M07, 22D05, 22E40 

\maketitle

\thispagestyle{empty}

\section{Introduction}

In this paper we study locally finite CW complexes $X_{m,n}$ which
are combinatorial models for the Baumslag--Solitar groups
$BS(m,n)$. The group $BS(m,n)$ acts freely and cocompactly on
$X_{m,n}$ with quotient space $Z_{m,n}$ homeomorphic to the standard
presentation $2$--complex of $BS(m,n)$. We wish to 
understand lattices in the locally compact group $\Aut(X_{m,n})$ of
combinatorial automorphisms. $BS(m,n)$ is one such
lattice but we are interested in others. 

Bass and Kulkarni have shown that if $T$ is a locally finite tree then 
any two uniform lattices in $\Aut(T)$ are commensurable up to conjugacy
\cite{basskulkarni}, so in a sense there is only ``one'' lattice. On the other
hand, if $T_1$ and $T_2$ are non-isomorphic locally finite trees then
$\Aut(T_1 \times T_2) \cong \Aut(T_1) \times \Aut(T_2)$ has a very
rich lattice theory, including the celebrated examples of Burger and
Mozes \cite{burgermozes} and Wise \cite{wisethesis,wisecsc}. 

The group $\Aut(X_{m,n})$ has some similarities with both $\Aut(T)$
and $\Aut(T_1) \times \Aut(T_2)$, and can be viewed as lying
intermediate between the two. Let $T_{m,n}$ denote the regular
directed tree in which every vertex has $m$ incoming and $n$ outgoing
edges. This is the usual Bass-Serre tree for $BS(m,n)$. The complex
$X_{m,n}$ is homeomorphic to $T_{m,n} \times \R$, with projection
inducing a homomorphism $\pi_*\co \Aut(X_{m,n}) \to
\Aut(T_{m,n})$. This homomorphism is continuous and \emph{injective}
when $m\not= n$ (though not an embedding; see Remark
\ref{notembedding}). Thus, while an element of $\Aut(T_1 \times T_2)$
is an arbitrary pair of tree automorphisms, an element of
$\Aut(X_{m,n})$ is just a single one.

Nevertheless, our main result shows that for many values of $m$ and
$n$, $\Aut(X_{m,n})$ is large enough to accommodate
incommensurable uniform lattices, just as $\Aut(T_1 \times T_2)$
does. Our main construction appears in Theorem \ref{mainthm}, restated
as follows. 

\begin{theorem}\label{main1}
If\/ $\gcd(k,n) \not= 1$ then $\Aut(X_{k,kn})$ contains uniform lattices
$G_1, G_2$ that are not abstractly commensurable. 
\end{theorem}

A similar statement holds for other cases of $\Aut(X_{k,kn})$, by a
different construction. Combining Theorems \ref{g3thm} and \ref{g4thm}
we have:

\begin{theorem}\label{main2}
If either 
\begin{enumerate}
\item $n$ has a non-trivial divisor $p\not= n$ such that $p < k$, or
\item $n < k$ and $k \equiv 1 \mod n$ 
\end{enumerate}
then\/ $\Aut(X_{k,kn})$ contains uniform lattices that are not abstractly
commensurable. 
\end{theorem}

The lattices in the previous results are all torsion-free. A 
combinatorial characterization of the torsion-free lattices in
$\Aut(X_{m,n})$ is given in Theorem \ref{gbslatticechar}. 

In the other direction, there is just one situation in which we can
say for sure that incommensurable lattices do not exist, namely, when
$\Aut(X_{m,n})$ is discrete. This occurs if and only if
$\gcd(m,n)=1$, by Theorem \ref{discrete}. 

It is perhaps not surprising that the incommensurable lattices of
Theorems \ref{main1} and \ref{main2} lie in $\Aut(X_{m,n})$ for which $m$
divides $n$. This behavior illustrates the large degree of symmetry
enjoyed by these complexes. By the same token, the groups $BS(k,kn)$
are unusually symmetric. When $k,\abs{n}>1$, the groups
$\Aut(BS(k,kn))$ and $\Out(BS(k,kn))$ fail to be finitely generated,
by \cite{collinslevin} (see also \cite{clay} for a newer, more
geometric proof). 

\subsection*{Methods}
Every torsion-free lattice in $\Aut(X_{m,n})$ is a generalized
Baumslag--Solitar group (GBS group). These are the fundamental
groups of graphs of groups in which all edge and vertex groups are
$\Z$. We make use of many of the standard tools for studying these
groups, such as deformations and covering theory for
finite-index subgroups.

The main new tool developed here is a commensurability invariant
called the \emph{depth profile}. It is a subset of $\N$ defined
relative to a choice of elliptic subgroup $V$.
It records the manner in which certain conjugates $V^g$ meet $V$. 
%some of the behavior of modulus-one hyperbolic elements of
%$G$. 
If two GBS groups are commensurable then their depth
profiles are the same, provided both groups contain $V$. To accommodate
changes in $V$, we impose an equivalence relation on the set of
subsets of $\N$. In this way we obtain a well defined 
invariant of GBS groups; see Section \ref{invariantsec}. Note that
this is an invariant of \emph{abstract} commensurability. 

Computing the depth
profile of a GBS group is not completely straightforward and we do it
only under certain assumptions. It becomes helpful to pass to finite
index subgroups in order to meet these assumptions; this is where the
covering theory comes in. 

One case where the depth profile can be computed easily is the case of 
Baumslag--Solitar groups. Using the depth profile we obtain a new and
much simpler proof of one of the main results of \cite{CRKZ}, in which
it is shown that certain Baumslag--Solitar groups are not %abstractly
commensurable.
It is precisely the most difficult case of that result 
that is handled easily using the depth profile. 
See Example \ref{bsexample} and Corollary \ref{crkz}. 

Section \ref{examplessec} is the heart of the paper, in which we
present the main examples and compute their depth profiles.
In Section \ref{furthersec} we present the additional examples 
comprising Theorem \ref{main2}. 

In Section \ref{cayleysec} we show that the incommensurable lattices
of Theorems \ref{main1} and \ref{main2} all admit isomorphic Cayley
graphs; see Corollary \ref{cayleycor}. This result is not completely 
obvious since their actions on $X_{m,n}$ are not transitive on
the vertices. However, it is enough that they act on the vertices with
the same orbits, which is what happens here. 

There are many basic questions remaining regarding lattices in
$\Aut(X_{m,n})$, such as determining the pairs $m,n$ for which
incommensurable lattices exist. We have said nothing about lattices
with torsion, either. Also, there is a phenomenon (see Figure
\ref{easycase} and Remark \ref{envelope}) of different groups
$\Aut(X_{m,n})$ and $\Aut(X_{m',n'})$ admitting isomorphic lattices in
an unexpected way. These questions are gathered together in Section
\ref{questions}. 

The motivation for this work came in part from the following
question of Stark and Woodhouse \cite[Question 1.9]{starkwoodhouse}: 
If $H$ and $H'$ are one-ended residually finite hyperbolic groups that
act geometrically on the same simplicial complex, are $H$ and $H'$
abstractly commensurable? The groups considered here are certainly not
hyperbolic or residually finite, but they provide another setting, 
in addition to lattices in products of trees, in which incommensurable
groups can have a common combinatorial model. See also
\cite{dergachevaklyachko} for another recent example of this
phenomenon. 

\begin{ack}
The author was partially supported by Simons Foundation award
\#638026. He thanks Alex Margolis for helpful comments and for
suggesting Theorem \ref{discrete}. He thanks an anonymous referee
for noticing a flaw in the original definition of depth profile,
and for several helpful comments. 
\end{ack}

\section{Preliminaries}\label{sec:prelim}

\subsection{Graphs}
A \emph{graph} $A$ is a pair of sets $(V(A),E(A))$ with maps
$\partial_0, \partial_1 \co E(A) \to V(A)$ and a free involution $e
\mapsto \overline{e}$ on $E(A)$, such that $\partial_i(\overline{e})
= \partial_{1-i}(e)$ for all $e$. An element of $E(A)$ is an oriented
edge with initial vertex $\partial_0(e)$ and terminal vertex
$\partial_1(e)$; then $\overline{e}$ is the ``same'' edge with the
opposite orientation. An edge is a \emph{loop} if $\partial_0(e)
= \partial_1(e)$. For each $v\in V(A)$ we define $E_0(v) = \{ e \in
E(A) \mid \partial_0(e) = v \}$. 

A \emph{directed graph} is a graph $A$ together with a partition $E(A)
= E^+(A) \sqcup E^-(A)$ that separates every pair $\{e,
\overline{e}\}$. The edges in $E^+(A)$ are called \emph{directed}
edges. The ``direction'' is from $\partial_0(e)$ to
$\partial_1(e)$. For each $v \in V(A)$ we define $E^+_0(v) = E^+(A)
\cap E_0(v)$ and $E^-_0(v) = E^-(A) \cap E_0(v)$.

\subsection{$G$--trees} 
By a \emph{$G$--tree} we mean a simplicial tree $X$ with an action of
$G$ by simplicial automorphisms, without inversions. Given $X$, if 
$g\in G$ fixes a vertex we call $g$ \emph{elliptic}. If it is not
elliptic, then there is a $g$--invariant line in $X$, called the
\emph{axis} of $g$, on which $g$ acts by a non-trivial
translation. In this case, we call $g$ \emph{hyperbolic}. 
 An \emph{elliptic subgroup} is a subgroup $H < G$ that
fixes a vertex.

\subsection{Automorphisms of CW complexes} 
Let $X$ be a CW complex. A \emph{topological automorphism} of $X$ is a
homeomorphism which preserves the cell structure, and in particular,
the partition of $X$ into open cells. The group of such automorphisms
will be denoted by $\Aut_{top}(X)$.

We are interested in a more combinatorial
notion, however. The \emph{combinatorial automorphism group},
denoted %simply
by $\Aut(X)$, is the quotient of $\Aut_{top}(X)$
in which two automorphisms are considered the same if they induce the
same permutation on the set of cells of $X$. 

In many cases one can lift $\Aut(X)$ to a subgroup of
$\Aut_{top}(X)$ by imposing a metric constraint on topological
automorphisms. For the complexes $X_{m,n}$ considered in this paper,
there are piecewise hyperbolic or Euclidean metrics available that can
be used in this way; see Section \ref{metricsect}.

\subsection{Locally finite complexes}
If $X$ is connected and locally finite then $\Aut(X)$ is a locally
compact group, as we  explain here. For any CW complex $X$ the
topology on $\Aut(X)$ has a subbasis given by the sets  
\[ U^X_{\sigma \to \tau} \ = \ \{ \, f \in \Aut(X) \mid f(\sigma) = \tau
  \, \}\]
for all pairs of cells $\sigma, \tau$ of $X$. A basis is given by
finite intersections of these sets. That is, two automorphisms are
close if they agree on a large finite collection of cells of $X$. 

Define two cells to be \emph{adjacent} if their closures
intersect nontrivially. For any cell $\sigma$ let $B_{\sigma}(r)$
denote the combinatorial ball of radius $r$ about $\sigma$. This is
defined recursively as follows: $B_{\sigma}(0) = \{\sigma\}$, and
$\tau \in B_{\sigma}(r)$ if $\tau$ is equal or adjacent to a cell in
$B_{\sigma}(r-1)$. In a locally finite CW complex, every cell is
adjacent to only finitely many cells, and each $B_{\sigma}(r)$ is
finite. If $X$ is connected then $X = \bigcup_r B_{\sigma}(r)$ for any
$\sigma$. 

Now let $G = \Aut(X)$ with $X$ connected and locally finite. If
$\sigma$ is a cell, its stabilizer $G_{\sigma}$ 
is open (it equals $U^X_{\sigma \to \sigma}$) and has the structure
\[ G_{\sigma} \ = \ \underset{r}{\lim_{\longleftarrow}} \,
\image\bigl(G_{\sigma} \overset{\text{res}}{\longrightarrow}
\Aut(B_{\sigma}(r))\bigr).  \]
As an inverse limit of finite groups, $G_{\sigma}$ is profinite. 
Since compact open subgroups are always commensurable, all
cell stabilizers are commensurable. Finally, the existence of compact
open subgroups implies that $\Aut(X)$ is locally compact.

\subsection{Lattices}
In a locally compact group $G$, a discrete subgroup $\Gamma < G$ is
called a \emph{lattice} if $G/\Gamma$ carries a finite positive
$G$-invariant measure, and a \emph{uniform lattice} if $G/\Gamma$ is
compact. 

Now let $X$ be a connected, locally finite CW
complex. Let $G = \Aut(X)$. A subgroup $\Gamma < G$ is discrete 
if and only if every cell stabilizer $\Gamma_{\sigma}$ is finite. In
this case, define the \emph{covolume} of $\Gamma$ to be
\[ \Vol(X/\Gamma) \ = \ \sum_{[\sigma] \in
  \cells(X/\Gamma)} 1/\abs{\Gamma_{\sigma}}.\]
The sum is taken over a set of representatives of the 
$\Gamma$--orbits of cells of $X$. 

The next result follows directly from \cite[1.5--1.6]{basslubotzky}. 
\nocite{basslubotzky}

\begin{proposition}\label{blprop}
  Let $X$ be a connected locally finite CW complex. 
  Suppose that $G = \Aut(X)$ acts cocompactly on $X$ and let $\Gamma <
  G$ be a discrete subgroup. Then
  \begin{enumerate}
  \item\label{bl1} $\Gamma$ is a lattice if and only if\/
    $\Vol(X/\Gamma)< \infty$ 
  \item\label{bl2} $\Gamma$ is a uniform lattice if and only if\/
    $X/\Gamma$ is compact. 
  \end{enumerate}
\end{proposition}
Note that it follows from \ref{blprop} that every torsion-free lattice
is uniform. (Indeed, every lattice with a bound on the size of finite
subgroups is uniform.)
The lattices considered in this paper will generally be torsion-free.

\section{Generalized Baumslag--Solitar groups}

\subsection{Definitions} 
A \emph{generalized Baumslag--Solitar group} (or \emph{GBS group}) is
a group that admits a 
graph of groups decomposition in which all vertex and edge groups are
$\Z$. Equivalently, it is a group that acts without inversions on a
simplicial tree such that all vertex and edge stabilzers are
infinite cyclic. Note that some authors require that GBS groups be
finitely generated. We have no such requirement here. 
%are not doing so here. 
%In this paper, we are not requiring that GBS groups be
%finitely generated. 
We refer to \cite{forester3,forester2,levitt-auto}
for general background on these groups. 

If $G$ is a GBS group with corresponding Bass--Serre tree $X$, we call
$X$ a \emph{GBS tree} for $G$. Its quotient graph of groups has every
edge and vertex group isomorphic to $\Z$, with inclusion maps given by
multiplication by various non-zero integers. Thus this graph of groups
is specified by a connected graph $A$ ($=X/G$) and a \emph{label
function} $\lambda \co E(A) \to (\Z-\{0\})$. We denote the
corresponding graph of groups by $(A,\lambda)_{\Z}$. Letting $\lambda
\co E(X) \to (\Z - \{0\})$ denote the induced label function on $X$,
we have 
\begin{equation}\label{labelindex}
\abs{\lambda(e)} \ = \ [G_{\partial_0(e)} : G_e]
\end{equation}
for all edges $e\in E(X)$.

\subsection{Fibered $2$--complexes}
It is often useful to take a topological viewpoint as in
\cite{scottwall}. A GBS group is then the fundamental group of the 
total space of a graph of spaces in which each edge and vertex space
is a circle. Given a labeled graph $(A,\lambda)$ there are oriented
circles $C_v$ for each vertex $v$ and $C_e = C_{\overline{e}}$ for
each $e \in E(A)$. For each $e$ let $M_e$ be the mapping cylinder of
the degree $\lambda(e)$ covering map $C_e \to C_{\partial_0(e)}$. Note
that $M_e$ contains embedded copies of both $C_e$ and
$C_{\partial_0(e)}$. The total space of the graph of spaces is now
defined as the quotient \[Z_{(A,\lambda)} \ = \ \bigsqcup_{e\in E(A)}
M_e \Big\slash \!\sim\] where all copies of $C_v$ (in $M_e$ for  $e
\in E_0(v)$) are identified by the identity map, and $C_e$ and
$C_{\overline{e}}$ (in $M_e$ and $M_{\overline{e}}$) are identified by
the identity, for all $v\in V(A)$ and  $e\in E(A)$. We will call
$Z_{(A,\lambda)}$ the \emph{fibered $2$--complex} associated to
$(A,\lambda)$. It is naturally foliated by circles, and may be thought
of as a $2$--dimensional analogue of a Seifert fibered space, with
singular fibers the vertex circles $C_v$. The leaf space of
$Z_{(A,\lambda)}$ is $A$ (realized topologically as a
$1$--complex). Each $M_e$ is embedded in $Z_{(A,\lambda)}$, and under
the map to the leaf space, $M_e$ is the pre-image of a closed
half-edge in $A$. 

\begin{example}\label{seifert} 
Let $M$ be a Seifert fibered $3$--manifold that is not closed. Let
$\Sigma$ be the quotient $2$--orbifold, which has only isolated cone
singularities. The underlying surface (also non-closed) contains a
\emph{spine}, i.e., a $1$--complex embedded as a deformation
retract. This spine $S$ may be chosen so that the singularities of
$\Sigma$ are all vertices of $S$. The union of the fibers over $S$ is
then a deformation retract of $M$, and it has the structure of a
fibered $2$--complex. Hence $\pi_1(M)$ is a GBS group. 
\end{example}

Other well known examples of GBS groups include free-by-cyclic groups
$F_n \rtimes_{\phi} \Z$ with periodic monodromy $\phi$ \cite{levitt1}
and one-relator groups with non-trivial center
\cite{pietrowski}. Every finite index subgroup of a GBS group is 
a GBS group. 

\begin{example}\label{bsname} 
Let $(A,\lambda)$ be the labeled graph having one vertex and one loop
$e$ with labels $\lambda(e) = m$, $\lambda(\overline{e}) = n$. The
corresponding GBS group is the \emph{Baumslag--Solitar group} 
\begin{equation}\label{presentation}
BS(m,n) \ = \ \langle \, a, t \mid ta^m
  t^{-1} = a^n\, \rangle.
\end{equation}
This is the \emph{standard}
GBS structure for $BS(m,n)$. In this case we give the space
$Z_{(A,\lambda)}$ the additional name $Z_{m,n}$. 
\end{example}

\begin{notation}
We denote by $\bigvee_{i=1}^k BS(m_i,n_i)$ the GBS group defined by
the labeled graph having one vertex and $k$ loops, each with labels
$m_i$ and $n_i$. 
\end{notation}

\subsection{Non-elementary GBS groups} 
A GBS group $G$ is \emph{elementary} if it is isomorphic to $\Z$,
$\Z\times \Z$, or the Klein bottle group, or is 
the union of an infinitely ascending chain of infinite cyclic groups 
$C_0 \subset C_1 \subset C_2 \subset \dotsm$. 
%$\dotsm \supset C_{2} \supset C_1 \supset C_{0}$. 
In the last case, $G$ is not finitely
generated. Being elementary characterizes the property that some (equivalently, 
every) GBS tree for $G$ has limit set of cardinality $0$, $1$, or
$2$. 
The three finitely generated cases correspond to limit sets of size $0$
or $2$ by \cite[Lemma 2.6]{forester2}. A GBS tree with
one-point limit set yields the given description of $G$ as an
ascending union by 
\cite[7.2]{bass} or 
\cite[(3.4)]{tits}. 

A fundmental property of GBS groups other than $\Z\times \Z$ and the
Klein bottle group (and therefore of all non-elementary GBS groups) is
that any two GBS trees for the same group $G$ define the same
partition of $G$ into elliptic and hyperbolic elements \cite[Lemmas
2.5, 2.6]{forester2}, and moreover have the same elliptic subgroups.

\subsection{Segments and index}
Let $X$ be a locally finite $G$--tree. An \emph{edge path} of length
$k$ is either a vertex $v_0$ (if $k=0$) or a sequence of edges $\sigma
= (e_1, \dotsc, e_k)$ with $\partial_1(e_i) = v_i =  
\partial_0(e_{i+1})$ for each $i$. The \emph{initial} and
\emph{terminal vertices} are $\partial_0(\sigma) = v_0$ and
$\partial_1(\sigma) = v_k$ respectively. The
\emph{reverse} of $\sigma$ is $\overline{\sigma} = (\overline{e}_k, \dots,
\overline{e}_1)$.

A \emph{segment} is an edge path with no
backtracking, meaning that $e_{i+1} \not= \overline{e}_i$ for $1 \leq
i < k$. We call a segment \emph{non-trivial} if its length is positive. 
The pointwise stabilizer of $\sigma$ is $G_{\sigma}
= G_{\partial_0 \sigma} \cap G_{\partial_1\sigma}$. If $x,y$ are
vertices of $X$ then $[x,y]$ denotes the segment with initial vertex
$x$ and terminal vertex $y$. The \emph{index} of $\sigma$ is the
number $i(\sigma) \ = \ [G_{\partial_0 \sigma}: G_{\sigma}]$. This
index is also defined for edges, since an edge is a segment of length
one. Note that length zero segments have index $1$.

If $X$ is a GBS tree with label function $\lambda$ induced by the
quotient graph of groups, then $i(e) = \abs{\lambda(e)}$ for all $e
\in E(X)$, as noted in \eqref{labelindex}. The main difference between
\emph{labels} and \emph{indices} is that labels may be negative. 

In order to compute indices of segments we will use the following
result, which is Corollary 3.6 of \cite{forester3}. The original
statement had edges and labels instead of segments and indices, but the
proof is the same. 

\begin{lemma}\label{fixpathlemma}
Let $X$ be a GBS tree and suppose the segment $\sigma$ is a
concatenation of non-trivial segments $\sigma = \sigma_1 \dotsm \sigma_k$. Let
$n_i = i(\sigma_i)$ and $m_i = i(\overline{\sigma}_i)$ for each
$i$. Then $G_{\sigma_1}$ fixes $\sigma$ if and only if 
\begin{equation}\label{fixpath}
  \prod_{i=2}^{\ell} n_i \ \text{ divides } \ \prod_{i=1}^{\ell - 1} m_i
\end{equation}
for all $\ell = 2, \dotsc, k$. \qed
\end{lemma}

Note: the lemma does not apply to every concatenation of
segments. The concatenation itself must also be a segment. 

%For the lemma to hold it is important that $\sigma$ itself is a
%segment, i.e. that there is no backtracking produced when
%concatenating the segments $\sigma_i$. 

The index of a segment can now be determined as follows (for instance,
by taking each $\sigma_i$ to be an edge): 

\begin{lemma}\label{segmentindex} 
Let $X$ be a GBS tree and suppose the segment $\sigma$ is a
concatenation of non-trivial segments $\sigma = \sigma_1 \dotsm \sigma_k$. Let
$n_i = i(\sigma_i)$ and $m_i = i(\overline{\sigma}_i)$ for each
$i$. Then $i(\sigma)$ is  the smallest positive integer $r$ such that 
\begin{equation}\label{pathindex}
  n_1 \mid r \ \text{ and } \ \prod_{i=1}^{\ell} n_i \ \text{
    divides } \ r \prod_{i=1}^{\ell - 1} m_i \ \text{ for } \ \ell = 2,
  \dotsc, k. 
\end{equation}
\end{lemma}

\begin{proof}
Write $G_{\partial_0 \sigma} = \langle \, x \, \rangle$. We want $r$
such that $G_{\sigma} = \langle \, x^r \, \rangle$, which occurs if
and only if $r$ is smallest such that $\langle \, x^r \, \rangle$
fixes $\sigma$. Consider a longer segment $\sigma_0 \cdot \sigma$
where $i(\overline{\sigma}_0) = r$ (perhaps in a larger GBS
tree). Then $G_{\sigma_0} = \langle \, x^r \, \rangle$, and so
$\langle \, x^r \, \rangle$ fixes $\sigma$ if and only if condition
\eqref{fixpath} holds for the concatenation $\sigma_0 \dotsm \sigma_k$
(for $\ell = 1, \dotsc, k$). These conditions are exactly statement
\eqref{pathindex}. 
\end{proof}

\begin{remark}\label{shortindex}
When $k=2$ there is a closed formula  $i(\sigma) = n_1 n_2 /
\gcd(m_1, n_2))$. Applying this formula iteratively is usually easier
than using \eqref{pathindex}. 
\end{remark}

More important than \eqref{pathindex}, perhaps, is %the fact
that $i(\sigma)$ and $i(\overline{\sigma})$ are determined completely
by the indices seen along $\sigma$, in any expression as a
concatenation. 

\begin{lemma}\label{hypelement}
Let $X$ be a GBS tree and suppose $\sigma = [x,gx]$ for some vertex
$x$ with $x \not= gx$. If\/ $i(\sigma) > 1$ then there is a hyperbolic
element $h\in G$ such that $hx = gx$. 
\end{lemma}

\begin{proof}
If $g$ is hyperbolic there is nothing to prove. So suppose $g$ is
elliptic with fixed subtree $X_g$. Then $X_g 
\cap \sigma = \{m\}$, the midpoint of $\sigma$. Let $G_x = \langle a
\rangle$. If $a$ fixes $m$ then $G_x \subset G_m$, and since $g\in
G_m$ we have $G_x = (G_x)^g = G_{gx}$. But then $i(\sigma) = 1$, a
contradiction. So $X_g \cap X_a = \emptyset$ and therefore $ga$ is a
hyperbolic element taking $x$ to $gx$. 
%(by \cite[Lemma 2.8]{forester1}, say). 
\end{proof}

\subsection{The modular homomorphism} 
Let $G$ be a GBS group with GBS tree $X$ and quotient
labeled graph $(A,\lambda)$. Fix a non-trivial elliptic element $a \in
G$. Since elliptic elements are all commensurable, every element $g$
satisfies a relation $g^{-1}a^mg = a^n $ for 
some non-zero integers $m,n$. The function $q(g) = m/n$ defines a
homomorphism $G \to \Q^{\times}$ called the \emph{modular
homomorphism} (cf. \cite{kropholler-centrality,basskulkarni}). It does
not depend on the choice of $a$.

The modular homomorphism takes the value $1$ on every elliptic
element, so it factors through $\pi_1(A)$, indeed through $H_1(A)$
since $\Q^{\times}$ is abelian. It can be computed from $(A,\lambda)$
as follows. If $g\in G$ maps to $\alpha \in H_1(A)$ represented by
the $1$--cycle $(e_1, \dotsc, e_n)$, then
\begin{equation}\label{mh}
  q(g) \ = \ \prod_{i=1}^n \lambda(e_i)/\lambda(\overline{e}_i).
\end{equation}
Recall that when $G$ is non-elementary, all GBS trees for $G$ have the
same elliptic elements. Thus, in this case, $q$ is well defined in
terms of $G$ alone. 

If $V$ is any non-trivial elliptic subgroup of $G$, then there is a
formula 
\begin{equation}\label{modulus}
  \abs{q(g)} \ = \ [V: V \cap V^g] / [V^g : V \cap V^g] 
\end{equation}
(see \cite[Section 6]{forester3}). Note that for $V = G_x$, the right
hand side is the ratio of indices $i(\sigma) / i(\overline{\sigma})$
for the segment $\sigma = [x, gx]$.

\subsection{The orientation character} 
Let $G$ be a GBS group with GBS tree $X$. The \emph{orientation
character} of $G$ (which a priori depends on $X$) is a homomorphism
$\chi\co G \to \{\pm 1\}$ defined by 
\[\chi(g) \ = \ q(g) / \abs{q(g)}.\] 
When $G$ is non-elementary $\chi$ is well defined in terms of $G$
alone, since this is true of $q$. When $G$ is $\Z$, $\Z\times \Z$, or
$\bigcup_i C_i$,  
every orientation character is trivial. In the case of the Klein
bottle group, there are two deformation spaces of GBS trees, one with
trivial and one with non-trivial orientation character. Note that
every GBS group has a subgroup of index at most $2$ with trivial
orientation character. 

\begin{remark}\label{signchange}
Let $G$ be a GBS group with GBS tree $X$ and quotient labeled graph
$(A,\lambda)$. If the orientation character is trivial then the
label function $\lambda$ can be made positive by \emph{admissible sign
changes}; see \cite[Lemma 2.7]{clayforester}. 

An admissible sign change is the change in $\lambda$ that results from
changing the choices of generators in the graph of groups defined by
$(A,\lambda)$. The graph of groups itself does not change, only its
description in terms of labels. If one changes the generator of a
vertex group $G_v$, then the labels of edges in $E_0(v)$ all change
sign. If one changes the generator of an edge group $G_e$, then
$\lambda(e)$ and $\lambda(\overline{e})$ both change sign. 
\end{remark}

\subsection{Labeled graphs and deformations}
It often happens that different labeled graphs define isomorphic
GBS groups. A given GBS group generally has no preferred GBS tree (or
labeled graph). Indeed, it is still an open problem to find an
algorithm which determines whether two labeled graphs define the same
group. (This problem has been solved in some special cases; see
especially \cite{forester3, levitt-auto, clayforester, dudkin, CRKZ}.) 

It is true, however, that any two cocompact GBS trees for a
non-elementary $G$ are related by an  \emph{elementary deformation}
\cite{forester1}. This means that they are related by a finite
sequence of elementary moves, called elementary collapses and
expansions. %There is also a \emph{slide move}, which can be expressed
%as an expansion followed by a collapse. 

\begin{definition}
Let $X$ be a $G$--tree and $e$ an edge of $X$ with endpoints in
distinct orbits, such that $G_e = G_{\partial_0 (e)}$. One obtains a new
$G$--tree $Y$ by collapsing each component of the subforest spanned by
$Ge$ to a vertex. This operation is called an \emph{elementary
  collapse} move. The vertex which is the image of $e$ has stabilizer
$G_{\partial_1(e)}$ and the set of elliptic subgroups of $G$ does not
change. 

In the setting of GBS trees it is convenient to use a slightly
restricted definition of expansion move (compared to \cite{forester1}, 
where it is just the reverse of a collapse move). 
Let $X$ and $Y$ be as above. The transition from $Y$ to $X$ is called an
\emph{elementary expansion} move in all cases \emph{except} when the
vertex $\partial_0(e)$ has valence one and has trivial stabilizer. 

The effect of an elementary move on the
quotient graph of groups is as follows (with $A=G_{\partial_1(e)}$ and 
$C = G_e = G_{\partial_0(e)}$): 

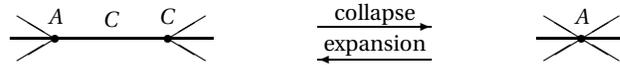
\begin{pict}{90}{13}
\put(10,5){\circle*{1}}
\put(25,5){\circle*{1}}

\put(10,5){\line(1,0){15}}

\thinlines
\put(45,6.5){\vector(1,0){15}}
\put(60,2.2){\vector(-1,0){15}}
\zindex{52.5}{8.2}{\mbox{collapse}}
\zindex{52.5}{4}{\mbox{expansion}}

% \put(25,5){\line(3,5){3}}
\put(25,5){\line(5,-3){5}}
\put(25,5){\line(5,3){5}}

\put(10,5){\line(-5,3){5}}
\put(10,5){\line(-5,-3){5}}

\put(80,5){\circle*{1}}
\put(80,5){\line(-5,3){5}}
\put(80,5){\line(-5,-3){5}}

% \put(80,5){\line(3,5){3}}
\put(80,5){\line(5,-3){5}}
\put(80,5){\line(5,3){5}}

\put(4,5){\line(1,0){27}}
\put(74,5){\line(1,0){12}}

\zindex{10}{8}{A}
\zindex{17.5}{7.5}{C}
\zindex{25}{8}{C}

\zindex{80}{8}{A}
\end{pict}% 

The elementary deformation shown below is called a \emph{slide move}. To
perform the move it is required that $D \subseteq C$. The edge with label
$C$ is allowed be a loop. 

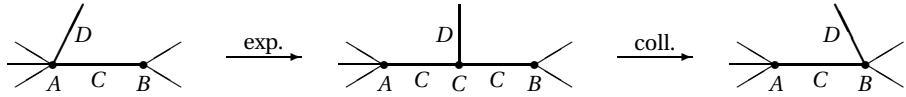
\begin{pict}{120}{16}
\thicklines
\put(102,4){\circle*{1}}
\put(114,4){\circle*{1}}
\put(102,4){\line(1,0){12}}
\put(114,4){\line(-1,2){4}}

\zindex{102}{1.5}{A}
\zindex{114}{1.5}{B}
\zindex{108}{2}{C}
\zindex{109.5}{8}{D}

\thinlines
\put(114,4){\line(5,3){5}}
\put(114,4){\line(5,-3){5}}
\put(102,4){\line(-5,3){5}}
\put(96,4){\line(1,0){6}}
\put(102,4){\line(-5,-3){5}}

%%%%%%%%

\thicklines
\put(6,4){\circle*{1}}
\put(18,4){\circle*{1}}
\put(6,4){\line(1,0){12}}
\put(6,4){\line(1,2){4}}

\zindex{6}{1.5}{A}
\zindex{18}{1.5}{B}
\zindex{12}{2}{C}
\zindex{10}{8}{D}

\thinlines
\put(18,4){\line(5,3){5}}
\put(18,4){\line(5,-3){5}}
\put(6,4){\line(-5,3){5}}
\put(0,4){\line(1,0){6}}
\put(6,4){\line(-5,-3){5}}

%%%%%%%%

\thicklines
\put(50,4){\circle*{1}}
\put(60,4){\circle*{1}}
\put(70,4){\circle*{1}}
\put(50,4){\line(1,0){20}}
\put(60,4){\line(0,1){8}}

\zindex{50}{1.5}{A}
\zindex{55}{2}{C}
\zindex{60}{1.5}{C}
\zindex{65}{2}{C}
\zindex{70}{1.5}{B}
\zindex{58}{8}{D}

\thinlines
\put(70,4){\line(5,3){5}}
\put(70,4){\line(5,-3){5}}
\put(50,4){\line(-5,3){5}}
\put(44,4){\line(1,0){6}}
\put(50,4){\line(-5,-3){5}}

%%%%%%%%

\zindex{34}{6.5}{\mbox{exp.}}
\zindex{86}{7}{\mbox{coll.}}
\put(29,5){\vector(1,0){10}}
\put(81,5){\vector(1,0){10}}
\end{pict}%
\end{definition}

With the definitions given above, any elementary move performed on a
GBS tree results in a GBS tree. The next proposition gives a
description of these moves in terms of labeled graphs. 

\begin{proposition}[Prop. 2.4 of \cite{forester3}] \label{gbsmoves} 
If an elementary move is performed on a generalized Baumslag--Solitar
tree, then the quotient graph of groups changes locally as follows: 

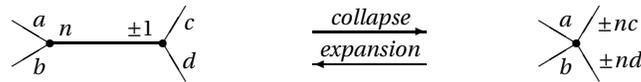
\begin{pict}{90}{12}
\thicklines
\put(10,5){\circle*{1}}
\put(25,5){\circle*{1}}

\put(10,5){\line(1,0){15}}

\thinlines
\put(45,6.5){\vector(1,0){15}}
\put(60,2.2){\vector(-1,0){15}}
\zindex{52.5}{8.2}{\mbox{collapse}}
\zindex{52.5}{4}{\mbox{expansion}}

\put(25,5){\line(3,5){3}}
\put(25,5){\line(3,-5){3}}

\put(10,5){\line(-5,3){5}}
\put(10,5){\line(-5,-3){5}}

\put(80,5){\circle*{1}}
\put(80,5){\line(-5,3){5}}
\put(80,5){\line(-5,-3){5}}

\put(80,5){\line(3,5){3}}
\put(80,5){\line(3,-5){3}}

\scriptsize
\zindex{8.5}{8}{a}
\zindex{8.5}{2}{b}
\zindex{12}{6.8}{n}
\zindex{22}{7}{\pm 1}
\zindex{28.5}{7.5}{c}
\zindex{28.5}{2.8}{d}

\zindex{78.5}{8}{a}
\zindex{78.5}{2}{b}
\zindex{85.5}{7.5}{\pm nc}
\zindex{85.9}{2.8}{\pm nd}
\end{pict} 
A slide move has the following description: 

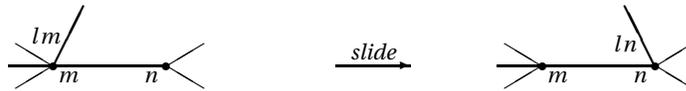
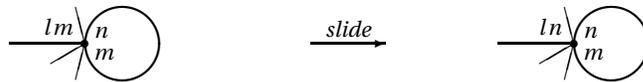
\begin{pict}{100}{13}
\thicklines
\put(75,3){\circle*{1}}
\put(90,3){\circle*{1}}
\put(75,3){\line(1,0){15}}
\put(90,3){\line(-1,2){4}}

\thinlines
\put(47.5,3){\vector(1,0){10}}
\zindex{52.5}{5}{\mbox{slide}}

\put(90,3){\line(5,3){5}}
\put(90,3){\line(5,-3){5}}
\put(75,3){\line(-5,3){5}}
\put(69,3){\line(1,0){6}}
\put(75,3){\line(-5,-3){5}}

\scriptsize
\zindex{77}{1.5}{m}
\zindex{88}{1.5}{n}
\zindex{86}{6}{ln}

\thicklines
\put(10,3){\circle*{1}}
\put(25,3){\circle*{1}}
\put(10,3){\line(1,0){15}}
\put(10,3){\line(1,2){4}}

\thinlines
\put(25,3){\line(5,3){5}}
\put(25,3){\line(5,-3){5}}
\put(10,3){\line(-5,3){5}}
\put(4,3){\line(1,0){6}}
\put(10,3){\line(-5,-3){5}}

\scriptsize
\zindex{12}{1.5}{m}
\zindex{23}{1.5}{n}
\zindex{9}{7}{lm}
\end{pict} 
or 

\begin{pict}{100}{10}
\thicklines

\put(82.5,5){\circle*{1}}
\put(87.5,5){\circle{10}}
\put(72.5,5){\line(1,0){10}}

\thinlines
\put(47.5,5){\vector(1,0){10}}
\zindex{52.5}{7}{\mbox{slide}}

\put(82.5,5){\line(-5,-3){4.5}}
\put(82.5,5){\line(-1,-4){1.2}}
\put(82.5,5){\line(-1,4){1.2}}

\scriptsize
\zindex{85.1}{3.5}{m}
\zindex{84.7}{6.5}{n}
\zindex{79.4}{7.3}{ln}

\thicklines

\put(17.5,5){\circle*{1}}
\put(22.5,5){\circle{10}}
\put(7.5,5){\line(1,0){10}}

\thinlines
\put(17.5,5){\line(-5,-3){4.5}}
\put(17.5,5){\line(-1,-4){1.2}}
\put(17.5,5){\line(-1,4){1.2}}

\scriptsize
\zindex{20.1}{3.5}{m}
\zindex{19.7}{6.5}{n}
\zindex{14}{7.3}{lm}
\end{pict} 
\end{proposition}

The moves above may be interpreted topologically, as homotopy
equivalences between fibered $2$--complexes in which annuli are
collapsed or expanded or slid over one another. 

\begin{remark}
The claim that any two cocompact GBS trees for a non-elementary $G$ are
related by an elementary deformation remains true with our restricted
notion of expansion move. The proofs given in \cite{forester1} and
\cite{forester3} produce deformations that do not use the forbidden
expansion move, since they come from folds, which can only increase
stabilizers. Alternatively, one may make both trees reduced and then
apply the main result of \cite{clayforester2}, which produces 
deformations of a restricted kind in which all trees are minimal. (The
forbidden move creates a non-minimal tree.) 
\end{remark}

\subsection{Subgroups of GBS groups}
Every subgroup of a GBS group is either a GBS group or a free
group. In the former case the inclusion $H \hookrightarrow G$ is
induced by a covering map from one fibered $2$--complex to
another. Such covering maps are encoded by \emph{admissible branched
coverings} of labeled graphs, defined as follows (cf. \cite[Lemma
5.3]{levitt1}). 

\begin{definition}\label{admissibledef}
An \emph{admissible branched covering} of labeled graphs, from
$(A,\lambda)$ to $(B,\mu)$, consists of a surjective graph morphism $
p \co A \to B$ between connected graphs together with a \emph{degree}
function 
\[d \co E(A) \sqcup V(A) \to \N\]
satisfying $d(e) = d(\overline{e})$ for all $e\in E(A)$ such that the
following holds. Given an edge $e \in E(B)$ with $\partial_0(e) = v$ 
and a vertex $u\in p^{-1}(v)$, let $k_{u,e} = \gcd(d(u),\mu(e))$. Then
\begin{enumerate}
\item $\abs{p^{-1}(e) \cap E_0(u)} = k_{u,e}$
\item $\lambda(e') = \mu(e)/k_{u,e}$ for each edge $e' \in p^{-1}(e)
  \cap E_0(u)$ 
\item $d(e') = d(u)/k_{u,e}$ for each edge $e' \in p^{-1}(e) \cap
  E_0(u)$. 
\end{enumerate}
See Figure \ref{admissiblefig}. 
\end{definition}

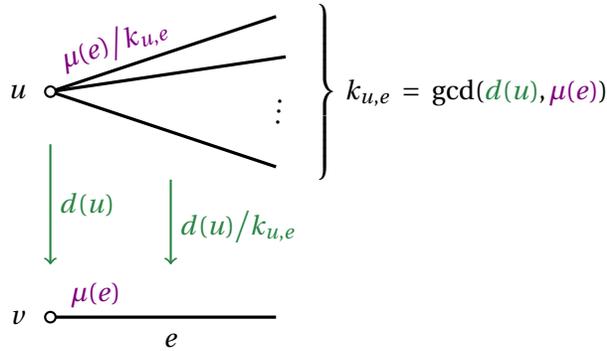
\begin{figure}[!ht]
\hspace*{3.5cm}
\begin{tikzpicture}
\small
  \draw[very thick] (1,4) -- (4,5);
  \draw[very thick] (1,4) -- (4,3);
  \draw[rotate around={-10:(1,4)},very thick] (1,4) -- (4,5);
  \filldraw[fill=white,thick] (1,4) circle (.7mm);
  \draw (0.8,4) node[anchor=east] {$u$};
  \draw[rotate around={-5:(1,4)}] (4.05,4.112) node[transform
  shape,anchor=center] {$\vdots$}; 
  \draw (4,4) node[anchor=west] {$\begin{rcases}  & \\ & \\ &
      \\ & \end{rcases} \ k_{u,e} \ = \
    \gcd(\textcolor{mygreen}{d(u)},\textcolor{violet}{\mu(e)})$}; 
  \draw[rotate around={18.435:(1,4)}] (1.15,4) node[transform
    shape,anchor=south west,violet] {$\mu(e)\big\slash k_{u,e}$};

  \draw[thick,->,color=mygreen] (1,3.3) -- (1,1.7);
  \draw[color=mygreen] (1,2.5) node[anchor=west] {$d(u)$}; 
  \draw[thick,->,color=mygreen] (2.6,2.83) -- (2.6,1.7);
  \draw[color=mygreen] (2.6,2.2) node[anchor=west] {$d(u)\big\slash
    k_{u,e}$};  
  
  \draw[very thick] (1,1) -- (4,1); 
  \filldraw[fill=white,thick] (1,1) circle (.7mm);
  \draw (0.8,1) node[anchor=east] {$v$};
  \draw (2.6,0.9) node[anchor=north] {$e$};
  \draw (1.15,1) node[anchor=south west,violet] {$\mu(e)$};
\end{tikzpicture}
\caption{The admissibility condition. Each edge of $p^{-1}(e) \cap
  E_0(u)$ has label $\mu(e)/k_{u,e}$ and degree $d(u)/k_{u,e}$. There
  are $k_{u,e}$ such edges.}\label{admissiblefig} 
\end{figure}

By lifting the graph of spaces structure, every covering space of a
fibered $2$--complex has a compatible graph of spaces structure, where
either every vertex and edge space is a line, or every vertex and
edge space is a circle. The two cases correspond to whether the
subgroup acts freely on the Bass--Serre tree
or is a GBS group (acting elliptically if it is cyclic); see \cite[Lemmas
  2.6--2.7]{forester2}.
Note that there is an induced surjective morphism $p \co A \to B$ of
underlying graphs, and the covering map is fiber-preserving, with $p$
the induced map on leaf spaces. 

\begin{proposition}\label{admissibleprop}
Let $G$ be a GBS group with labeled graph $(B,\mu)$. There is a
one-to-one correspondence between conjugacy classes of GBS 
subgroups of $G$ (excluding hyperbolic cyclic subgroups) and
admissible branched coverings $(A,\lambda) \to (B,\mu)$.  
\end{proposition}

This result is %essentially the same as
very similar to \cite[Lemma 6.3]{levitt1}, but
we explain it here somewhat differently. 

\begin{proof}
It suffices to classify the fibered $2$--complex covering spaces of
$Z_{(B,\mu)}$. Recall that $Z_{(B,\mu)}$ is a union of mapping
cylinders $M_e$. The admissibility condition is simply a description
of the finite-sheeted covers of these subspaces $M_e$. As a mapping
cylinder, $M_e$ deformation retracts onto $C_{\partial_0(e)}$, and has
infinite cyclic fundamental group. However it is also a fiber bundle
over this circle, giving it the structure of a mapping
\emph{torus}. The fiber is a cone on $n = \abs{\mu(e)}$ points,
denoted by $C(n)$, and the monodromy $\phi\co C(n) \to C(n)$ is the
automorphism which permutes the $n$ points by an $n$--cycle
$\sigma$. Let $z_1, \dotsc, z_n$ be the $n$ points and $c$ the cone
point. Writing $M_e$ as 
\[C(n) \times [0,1] \big\slash (x,0) \sim (\phi(x),1),\]
the image of $c \times [0,1]$ is the singular circle
$C_{\partial_0(e)}$ and the images of $z_i \times [0,1]$ join up to
form the circle $C_e$. The map $C_e \to C_{\partial_0(e)}$ in which
each $z_i \times \{t\}$  maps to $c\times \{t\}$ is the degree
$\mu(e)$ covering map defining $M_e$. 

For any $d\in \N$, the $d$--sheeted covering space $N$ of $M_e$ is the
mapping torus of 
\[\phi^d \co C(n) \to C(n).\]
The permutation $\sigma^d$ decomposes into $k = \gcd(d,n)$ disjoint
$(n/k)$--cycles. Thus the pre-image of $C_e$ in $N$ is $k$ disjoint
circles and $N$ is the mapping cylinder of the map $\bigsqcup_{i=1}^k
S^1 \to S^1$ in which each component maps by degree $\mu(e)/k$. Each
of the circles above $C_e$ covers with degree $d/k$, and the circle
above $C_{\partial_0(e)}$ covers with degree $d$. Now $N$ is exactly
described by the admissibility condition as shown in Figure
\ref{admissiblefig}, with $d = d(u)$. Thus for any covering map of
fibered $2$--complexes, the induced morphism of indexed graphs is an
admissible branched covering. The degree function records the degrees
of each fiber in the cover mapping to its image. (Orientations of the
fibers in the cover are chosen so that these degrees are all
positive.) 

Conversely, let $p \co (A,\lambda) \to (B,\mu)$ be an admissible
branched covering. Recall that $Z_{(A,\lambda)}$ and $Z_{(B,\mu)}$
are, respectively, unions of the subspaces $M_{e'}$ ($e' \in E(A)$)
and $M_e$ ($e \in E(B)$). For each $e \in E(B)$ with $\partial_0(e) =
v$ and each $u \in p^{-1}(v)$ the subspace
\[N_{u,e} \ = \ \bigcup_{e' \in p^{-1}(e) \cap E_0(u)} M_{e'} \]
of $Z_{(A,\lambda)}$ has a degree $d(u)$ covering map $p_{u,e} \co
N_{u,e} \to M_e$, by the admissibility condition. These subspaces
$N_{u,e}$ intersect each other only along the circles $C_u$ and
$C_{e'}$. The restrictions $p_{u,e} \vert_{C_u}$ and $p_{u,e}
\vert_{C_{e'}}$ are coverings of degrees $d(u)$ and $d(e')$
respectively. By adjusting the covering maps on $N_{u,e}$ by
fiber-preserving isotopies, the restrictions of the coverings to the
circles $C_u$ and $C_{e'}$ (for $u \in V(A)$, $e' \in E(A)$) can be
made to all agree on each circle. Then the covering maps $p_{u,e}$
join to give a covering of fibered $2$--complexes. 
\end{proof}

\begin{remark}
Every \emph{finite index} subgroup of $G$ is a GBS subgroup, and these
correspond to the branched coverings $(A,\lambda) \to (B,\mu)$ for
which the morphism $A \to B$ has finite fibers. More generally, if
$H<G$ corresponds to the branched covering $p\co (A,\lambda)  \to
(B,\mu)$, then $[G:H] \ = \ \sum_{u \in p^{-1}(v)} d(u)$, for any
vertex $v$ of $B$. 
\end{remark}

\begin{remark}
Every GBS group is \emph{coherent}, meaning that every finitely
generated subgroup is finitely presented. Since every subgroup is free
or GBS, it suffices to note that a GBS group $G$ with minimal GBS tree
$X$ is finitely generated if and only if $X$ is cocompact. In that
case, $G$ is the fundamental group of a compact fibered $2$--complex,
and is finitely presented. 
\end{remark}

\section{The $2$--complexes $X_{m,n}$ and $X_{(A,\lambda)}$}
\label{complexes-sec} 
Fix positive integers $m,n$ and let $k = \gcd(m,n)$. Let $T_{m,n}$
denote the directed simplicial tree in which every vertex has $m$
incoming edges and $n$ outgoing edges. This tree is the Bass--Serre
tree of $BS(m,n)$ with its standard labeled graph. The orientations
are such that the stable letter $t$ in the presentation
\eqref{presentation} is directed forward, and has axis in $T_{m,n}$
that is a directed line which $t$ shifts forward one unit. 

Recall from Example \ref{bsname} the fibered $2$--complex $Z_{m,n}$.
It is homeomorphic to the presentation $2$--complex for the
presentation of $BS(m,n)$ given in \eqref{presentation}. That CW
complex has one vertex, two edges labeled $a$ and $t$, and one
$2$--cell attached by the boundary word $t a^m t^{-1} a^{-n}$. The
cell structure we want to put on $Z_{m,n}$ is obtained from this by
subdivision and is illustrated in Figure \ref{Zcells}. The initial
$2$--cell  has been split into $k$ $2$--cells, with new $1$--cells
labeled $t_1, \dots, t_{k-1}$. The new $2$--cells are attached by the
words $t_i a^{m/k} t^{-1}_{i+1} a^{-n/k}$ (indices modulo $k$), with
$t_0$ understood to mean $t$. There is still only one vertex. 

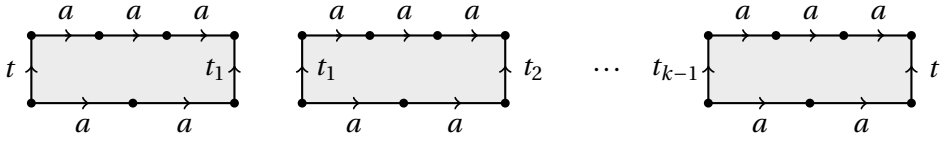
\begin{figure}[!ht]
\begin{tikzpicture}[scale=0.9]
\small

\filldraw[fill=gray!15,thick] (1,1) rectangle (4,2);
\filldraw[fill=gray!15,thick] (5,1) rectangle (8,2);
\filldraw[fill=gray!15,thick] (11,1) rectangle (14,2);

\filldraw[fill=black,thick] (1,1) circle (.5mm);
\filldraw[fill=black,thick] (2.5,1) circle (.5mm);
\filldraw[fill=black,thick] (4,1) circle (.5mm);
\filldraw[fill=black,thick] (5,1) circle (.5mm);
\filldraw[fill=black,thick] (6.5,1) circle (.5mm);
\filldraw[fill=black,thick] (8,1) circle (.5mm);
\filldraw[fill=black,thick] (11,1) circle (.5mm);
\filldraw[fill=black,thick] (12.5,1) circle (.5mm);
\filldraw[fill=black,thick] (14,1) circle (.5mm);

\filldraw[fill=black,thick] (1,2) circle (.5mm);
\filldraw[fill=black,thick] (2,2) circle (.5mm);
\filldraw[fill=black,thick] (3,2) circle (.5mm);
\filldraw[fill=black,thick] (4,2) circle (.5mm);
\filldraw[fill=black,thick] (5,2) circle (.5mm);
\filldraw[fill=black,thick] (6,2) circle (.5mm);
\filldraw[fill=black,thick] (7,2) circle (.5mm);
\filldraw[fill=black,thick] (8,2) circle (.5mm);
\filldraw[fill=black,thick] (11,2) circle (.5mm);
\filldraw[fill=black,thick] (12,2) circle (.5mm);
\filldraw[fill=black,thick] (13,2) circle (.5mm);
\filldraw[fill=black,thick] (14,2) circle (.5mm);

\draw[thick,->] (1,1.54) -- (1,1.55);
\draw[thick,->] (4,1.54) -- (4,1.55);
\draw[thick,->] (5,1.54) -- (5,1.55);
\draw[thick,->] (8,1.54) -- (8,1.55);
\draw[thick,->] (11,1.54) -- (11,1.55);
\draw[thick,->] (14,1.54) -- (14,1.55);

\draw[thick,->] (1.84,1) -- (1.85,1);
\draw[thick,->] (3.34,1) -- (3.35,1);
\draw[thick,->] (5.84,1) -- (5.85,1);
\draw[thick,->] (7.34,1) -- (7.35,1);
\draw[thick,->] (11.84,1) -- (11.85,1);
\draw[thick,->] (13.34,1) -- (13.35,1);

\draw[thick,->] (1.5,2) -- (1.6,2);
\draw[thick,->] (2.5,2) -- (2.6,2);
\draw[thick,->] (3.5,2) -- (3.6,2);
\draw[thick,->] (5.5,2) -- (5.6,2);
\draw[thick,->] (6.5,2) -- (6.6,2);
\draw[thick,->] (7.5,2) -- (7.6,2);
\draw[thick,->] (11.5,2) -- (11.6,2);
\draw[thick,->] (12.5,2) -- (12.6,2);
\draw[thick,->] (13.5,2) -- (13.6,2);

\draw (9.52,1.46) node {$\dotsm$};

\draw (1.75,0.9) node[anchor=north] {$a$};
\draw (3.25,0.9) node[anchor=north] {$a$};
\draw (5.75,0.9) node[anchor=north] {$a$};
\draw (7.25,0.9) node[anchor=north] {$a$};
\draw (11.75,0.9) node[anchor=north] {$a$};
\draw (13.25,0.9) node[anchor=north] {$a$};
\draw (1.5,2.1) node[anchor=south] {$a$};
\draw (2.5,2.1) node[anchor=south] {$a$};
\draw (3.5,2.1) node[anchor=south] {$a$};
\draw (5.5,2.1) node[anchor=south] {$a$};
\draw (6.5,2.1) node[anchor=south] {$a$};
\draw (7.5,2.1) node[anchor=south] {$a$};
\draw (11.5,2.1) node[anchor=south] {$a$};
\draw (12.5,2.1) node[anchor=south] {$a$};
\draw (13.5,2.1) node[anchor=south] {$a$};
\draw (0.9,1.5) node[anchor=east] {$t$};
\draw (4,1.5) node[anchor=east] {$t_1$};
\draw (5.05,1.5) node[anchor=west] {$t_1$};
\draw (8.1,1.5) node[anchor=west] {$t_2$};
\draw (11,1.5) node[anchor=east] {$t_{k-1}$};
\draw (14.1,1.5) node[anchor=west] {$t$};

\end{tikzpicture}
\caption{The cell structure for $Z_{m,n}$ when $m=3k$,
  $n=2k$. 
  }\label{Zcells}
\end{figure}

The $2$--complex $X_{m,n}$ is defined to the the universal cover of
$Z_{m,n}$ with the induced cell structure. It is tiled entirely by
quadrilateral $2$--cells of the kind shown in Figure \ref{Zcells},
with sides of length $1, m/k, 1, n/k$. The reason for subdividing the
initial cell structure of $Z_{m,n}$ is to increase the homogeneity of
$X_{m,n}$ and make it as symmetric as possible. 

We extend the definition to allow $m,n \in \Z-\{0\}$ by declaring
$X_{m,n} = X_{\abs{m},\abs{n}}$. Then $BS(m,n)$ is a lattice in
$\Aut(X_{m,n})$ for any $m,n$ (since $X_{m,n}$ is the universal cover
of $Z_{\pm m, \pm n}$). However, when discussing an unspecified
$X_{m,n}$, the default assumption will be that $m,n > 0$.

\subsection{Combinatorial description}\label{combinatorialsect}
The space $X_{m,n}$ is homeomorphic to %the product
$T_{m,n} \times \R$, but combinatorially and geometrically it is very
far from being a product (unless $m=n$). Let $\pi \co X_{m,n} \to
T_{m,n}$ be the projection map. The $1$--cells of $X_{m,n}$ will be
called \emph{horizontal} if they map to a vertex of $T_{m,n}$ and 
\emph{vertical} if they map to an edge. The vertical edges inherit 
orientations from the edges of $T_{m,n}$, consistent with the
orientations labeled $t$ or $t_i$  in Figure \ref{Zcells}. We view
this direction as the ``upward'' direction in $X_{m,n}$ and in
$T_{m,n}$. (The horizontal $1$--cells are \emph{not} oriented; the
orientations labeled $a$ in Figure \ref{Zcells} should be ignored when
considered as cells in $X_{m,n}$.) 

For any $v \in V(T_{m,n})$ the pre-image $\pi^{-1}(v)$ will be called
a \emph{branching line}. The pre-image of a closed edge of $T_{m,n}$
will be called a \emph{strip}. Note that the branching lines are tiled
by horizontal edges. We define an \emph{$(i,j)$--cell} to be a
$2$--cell attached by a combinatorial path consisting of one upward
vertical edge, $i$ horizontal edges, one downward vertical edge, and
$j$ horizontal edges. An \emph{$(i,j)$--strip} is a bi-infinite
sequence of $(i,j)$--cells, joined along their vertical edges. Every
strip in $X_{m,n}$ is a $(m/k,n/k)$--strip.

We may regard $X_{m,n}$ as being assembled from branching lines and
$(m/k,n/k)$--strips just as $T_{m,n}$ is made of vertices and
edges. Each branching line has $n$ strips above it and $m$ strips
below it. When attaching a strip \emph{above} a branching line, there
are $n/k$ ways to do this; if we identify the vertices of the
branching line with $\Z$, the vertical edges of the strip will meet a
coset $i + (n/k)\Z$ for some $i$. In $X_{m,n}$, the $n$ strips are
joined along the cosets $i + (n/k)\Z$ for $i = 1, \dotsc, n$. Thus
every vertex on the line has $k$ outgoing vertical edges. 

The strips below the branching line are attached in a similar manner;
there are $m$ of them, attached along the cosets $i + (m/k)\Z$ for $i
= 1, \dotsc m$. Each vertex in the line has $k$ incoming vertical
edges. This description of the neighborhood of every branching line,
together with the projection to $T_{m,n}$, completely determines
$X_{m,n}$ as a CW complex.

\subsection{Metric structure}\label{metricsect}
The complex $X_{m,n}$ admits a piecewise-Riemannian metric on which
$\Aut(X_{m,n})$ acts by isometries. The construction is based on
\cite{farbmosher}, which treats the case of $X_{1,n}$. If $m = n$ then
each $2$--cell is isometric to a Euclidean $m/k \times 1 = 1 \times 1$
rectangle, and all $1$--cells have length $1$. If $m \not= n$ then
each $2$--cell is isometric to a right-angled quadrilateral region in
the hyperbolic plane (a ``horobrick'') whose vertical sides are
geodesics of length $\abs{\log(m/n)}$ and whose horizontal sides are
concentric horocyclic arcs of lengths $m/k$ and $n/k$. 
This metric gives every horizontal $1$--cell length $1$. 
%Now every horizontal $1$--cell has length $1$. 

When $m \not= n$, each strip is isometric to the region in $\Hyp^2$
between two concentric horocycles of distance $\abs{\log(m/n)}$
apart. Since the branching lines are horocycles, $X_{m,n}$ has
concentrated positive curvature along these lines, and is definitely
not hyperbolic in any global sense. On the other hand, each strip in
$X_{m,m}$ is isometric to a Euclidean strip of width $1$, and
$X_{m,m}$ is isometric to $T_{m,m} \times \R$.

\subsection{Projection} 
The complex $X_{1,1}$ is the Euclidean plane tiled by unit squares. It
admits rotations and the branching lines are not invariant. In all
other cases, branching lines and strips in $X_{m,n}$ are preserved by
$\Aut(X_{m,n})$. Hence the projection $\pi \co X_{m,n} \to T_{m,n}$
induces a homomorphism $\pi_*\co \Aut(X_{m,n}) \to \Aut(T_{m,n})$. Now
choose consistent orientations for all of the branching lines. Every
$g \in \Aut(X_{m,n})$ either preserves the orientations of all
branching lines, or reverses all of them. We will call $g$ 
\emph{orientation preserving} or \emph{orientation reversing}
accordingly. 

\begin{lemma}\label{fixedray}
Suppose $m < n$. Let $\alpha$ be a directed ray in $T_{m,n}$ and
suppose $\pi_*(g)$ fixes $\alpha$ pointwise, for some orientation
preserving $g \in \Aut(X_{m,n})$. Then $g$ fixes $\pi^{-1}(\alpha)$
pointwise. 
\end{lemma}

A similar statement holds for ``anti-directed'' rays if $m > n$. 

\begin{proof}
The metric on $\pi^{-1}(\alpha)$ is isometric to a closed horoball in
$\Hyp^2$. Thus $g$ acts by a hyperbolic isometry preserving this
horoball. In the upper half plane model with the center of the
horoball at infinity, the only such isometries are reflections about
vertical lines and horizontal translations. Since $g$ is orientation
preserving, it is a translation. However, in this model, the
$2$--cells are rectangular regions of the form $[a,b] \times [c, d]$
with width $\abs{b-a}$ getting arbitrarily large as one moves upward
in the plane. No horizontal translation can preserve such a cell
structure, except for the identity. 
\end{proof}

\begin{corollary}\label{injective}
If $m \not= n$ then $\pi_*\co \Aut(X_{m,n}) \to \Aut(T_{m,n})$ is
continuous and injective. 
\end{corollary}

\begin{proof}
For continuity note that the pre-image of $U^{T_{m,n}}_{\sigma
\to \tau}$ is the union of the sets $U^{X_{m,n}}_{\tilde{\sigma} \to
\tilde{\tau}}$ for $\tilde{\sigma} \in \pi^{-1}(\sigma)$,
$\tilde{\tau}\in \pi^{-1}(\tau)$. For injectivity, suppose first that
$g \in \ker(\pi_*)$ is orientation preserving. Then $g = \id$ by Lemma
\ref{fixedray}, since $T_{m,n}$ is a union of rays of the relevant
type (directed or anti-directed). 

If $g$ is orientation reversing, then it acts on every branching line
as a reflection with a unique fixed point, and similarly on every
strip as a reflection. In the latter case, the line segment of fixed
points is either a vertical edge or it passes through the center of a
$2$--cell. If, say, $m < n$ and $n/k > 2$, consider two strips whose
lower sides are the same branching line, with vertical edges joined
along cosets one unit apart. Then the fixed point sets of the strips
cannot agree on the branching line and we have a contradiction. If
$n/k = 2, m/k=1$ then one finds a similar contradiction by considering
an arrangement of four strips whose projection to $T_{m,n}$ is a
segment of two downward edges followed by two upward edges. 
\end{proof}

\begin{remark}\label{notembedding}
The map $\pi_*$ is not an embedding. The topology on $\Aut(X_{m,n})$
is strictly finer than the subspace topology on
$\pi_*(\Aut(X_{m,n}))$. One can show that $\pi_*(U^{X_{m,n}}_{\sigma
\to \tau})$ is not open in $\pi_*(\Aut(X_{m,n}))$, for any cells
$\sigma, \tau$ with $U^{X_{m,n}}_{\sigma \to \tau} \not=\emptyset$. 

\begin{proof}[Proof in the case $\sigma = \tau$]
Note that $\pi_*(U^X_{\sigma \to \sigma})$ contains the identity
element of $\Aut(T) \cap \pi_*(\Aut(X))$. Every basic neighborhood of
the identity has the form $U^T_{\sigma_1 \to \sigma_1} \cap \dotsm
\cap U^T_{\sigma_{\ell} \to \sigma_{\ell}} \cap \pi_*(\Aut(X))$. Let
$S \subset T$ be a finite subtree containing $\pi(\sigma), \sigma_1,
\dotsc, \sigma_{\ell}$. The subcomplex $\pi^{-1}(S)$ admits a
non-trivial shift which projects to the identity on $S$.  This shift
extends to an automorphism $g \in \Aut(X_{m,n})$. Now $\pi_*(g) \in
U^T_{\sigma_1 \to \sigma_1} \cap \dotsm \cap U^T_{\sigma_{\ell} \to
\sigma_{\ell}} \cap \pi_*(\Aut(X))$ while $g \not\in U^X_{\sigma \to
\sigma}$. Since $\pi_*$ is injective, $\pi_*(g) \not\in
\pi_*(U^X_{\sigma \to \sigma})$. Hence no basic neighborhood of the
identity is contained in $\pi_*(U^X_{\sigma \to \sigma})$. 
\end{proof}
\end{remark}

\subsection{General labeled graphs}
For any labeled graph $(A,\lambda)$ let $T_{(A,\lambda)}$ denote the 
corresponding Bass--Serre tree. We will define $X_{(A,\lambda)}$ to be
the universal cover of $Z_{(A,\lambda)}$ with a suitable cell
structure. 

For each $v$ put a cell structure on the circle $C_v$ with one 
vertex and one edge. Given $e \in E(A)$, let $n_e = \abs{\lambda(e)}$,
$m_e = \abs{\lambda(\overline{e})}$, and $k_e = \gcd(m_e,n_e)$. The
annulus $M_e \cup M_{\overline{e}}$ has boundary curves of lengths
$n_e$ and $m_e$ attached to $C_{\partial_0(e)}$ and
$C_{\partial_1(e)}$ respectively. Tile this annulus with $k_e$
$(m_e/k_e, n_e/k_e)$--cells. Doing this for every edge we obtain a
cell structure for $Z_{(A,\lambda)}$, which then induces one on
$X_{(A,\lambda)}$. 

Every strip above the annulus $M_e \cup M_{\overline{e}}$ is an
$(m_e/k_e, n_e/k_e)$--strip, and $X_{(A,\lambda)}$ is
homeomorphic to $T_{(A,\lambda)} \times \R$. For each $e$, every
branching line covering $C_{\partial_0(e)}$ has $n_e$ strips
above it covering $M_e \cup M_{\overline{e}}$, attached along the
cosets $i + (n_e/k_e)\Z$ for $i = 1, \dotsc, n_e$. Every branching
line covering $\partial_1(e)$ has $m_e$ strips below it covering $M_e
\cup M_{\overline{e}}$, attached along the cosets $i + (m_e/k_e)\Z$
for $i = 1, \dotsc, m_e$. 

The space $X_{(A,\lambda)}$ admits an
$\Aut(X_{(A,\lambda)})$--invariant metric using the construction from 
Section \ref{metricsect}. 
%just as $X_{m,n}$
%does. 
Each $(m_e/k_e, n_e/k_e)$--cell is metrized as a horobrick of
height $\abs{\log(m_e/n_e)}$ with horocyclic sides of lengths $m_e/k_e$
and $n_e/k_e$ (if $m_e \not= n_e)$, or as a Euclidean $1 \times 1$
square (if $m_e = n_e$). The $(m_e/k_e, n_e/k_e)$--strips are then
horocyclic strips or Euclidean strips accordingly. This metric on
$X_{(A,\lambda)}$ is quasi-isometric, but not isometric, to the one
used by Whyte in \cite{whyte}. In the latter paper, each
$(m,n)$--horocyclic strip has width $1$ and constant curvature
$-\abs{\log(m/n)}$, whereas here it has width $\abs{\log(m/n)}$
and constant curvature $-1$. 

\begin{remark}
The notions of \emph{orientation preserving} and \emph{reversing}
automorphisms of $X_{m,n}$ apply equally well to $X_{(A,\lambda)}$
whenever $X_{(A,\lambda)} \not\cong X_{1,1}$. If $G$ is the GBS group
defined by $(A,\lambda)$, then these notions agree with the
orientation character on $G$. That is, $g\in G <
\Aut(X_{(A,\lambda)})$ is orientation preserving if and only if
$\chi(g) = 1$. 

In this way, the orientation character extends to a homomorphism $\chi
\co \Aut(X_{(A,\lambda)}) \to \{\pm 1\}$, even if the modular 
homomorphism $q$ does not. 
\end{remark}

\subsection{Torsion-free lattices in $X_{m,n}$}\label{latticesec}
One simple way for a GBS group to be a lattice in $\Aut(X_{m,n})$ is
if its labeled graph $(A, \lambda)$ satisfies $X_{(A,\lambda)} \cong
X_{m,n}$. The next result gives a criterion for this. 

\begin{proposition}\label{latticeprop}
Let G be the GBS group defined by $(A,\lambda)$ and suppose there is a
directed graph structure $E(A) = E^+(A) \sqcup E^-(A)$ on $A$ such
that 
\begin{enumerate}
\item\label{i1} for every $v \in V(A)$,
  \[ \sum_{e \in E^+_0(v)} \abs{\lambda(e)} \ = \ n \ \text{
    and } \ \sum_{e \in E^-_0(v)} \abs{\lambda(e)} \ = \ m \] 
\item\label{i2} for every $e \in E^+(A)$, let $n_e =
  \abs{\lambda(e)}$, $m_e = \abs{\lambda(\overline{e})}$, $k_e =
  \gcd(m_e, n_e)$, and $k = \gcd(m,n)$; then 
  \[ n_e / k_e \ = \ n/k
  \ \text{ and } \ m_e / k_e \ = \ m/k.\]
\end{enumerate}
Then $X_{(A,\lambda)} \cong X_{m,n}$, and hence $G$ is a lattice in\/
$\Aut(X_{m,n})$. 
\end{proposition}

\begin{proof}
Condition \eqref{i1} says that the tree $T_{(A,\lambda)}$, with
directed graph structure induced from $A$, is isomorphic to
$T_{m,n}$. This is so because the two sums count the numbers of strips
entering (resp. exiting) each branching line in
$X_{(A,\lambda)}$. Condition \eqref{i2} says that every strip in
$X_{(A,\lambda)}$ is a $(m/k,n/k)$--strip. It remains to examine how
these strips join the branching lines. 

Fix a vertex $v\in V(A)$ and a branching line $L$ covering $C_v$. For
each $e \in E^+(A) \cap E_0(v)$ there are $\abs{\lambda(e)}$ strips
above $L$ mapping to $M_e \cup M_{\overline{e}}$. These are
attached along the cosets $i + (n/k)\Z$ for $i = 1, \dotsc,
\abs{\lambda(e)}$. Each coset has the same number of such strips
attached to it (namely, $k_e$). Hence, overall, the $n$ strips above
$L$ are distributed evenly among the cosets of $(n/k)\Z$, with $k$ of
them attached along each one. The strips attached below $L$ are
also equidistributed among the cosets of $(m/k)\Z$. Now
$X_{(A,\lambda)}$ matches the description of $X_{m,n}$ from Section 
\ref{combinatorialsect}. 
\end{proof}

Now suppose $m\not= n$ and consider a general torsion-free uniform 
lattice $G$ in $\Aut(X_{m,n})$. It acts freely and cocompactly on
$X_{m,n}$ by Proposition \ref{blprop}. Note that every
automorphism preserves the directed structure on $T_{m,n}$ because
$m\not= n$. In particular, no strip has its sides exchanged, so every
strip covers an annulus in $X_{m,n}/G$. Each branching line covers a
circle. The quotient is then a compact fibered $2$--complex,
homeomorphic to some $Z_{(A,\lambda)}$. Moreover the graph $A$ is a
directed graph, with directed structure induced by that of $T_{m,n}$.

We put a cell structure on $Z_{(A,\lambda)}$ by identifying it with
$X_{m,n}/G$. There is a \emph{length function} $\ell \co V(A) \sqcup
E(A) \to \N$ defined as follows. For $v \in V(A)$, $\ell(v)$ is the
combinatorial length of the circle $C_v$. For $e \in E(A)$, $\ell(e)$
is the number of $2$--cells tiling the annulus $M_e \cup
M_{\overline{e}}$ (that is, its combinatorial girth). Note that
$\ell(e) = \ell(\overline{e})$ for all $e\in E(A)$. The edges in
$M_e\cup M_{\overline{e}}$ crossing from one boundary component to the
other are called \emph{vertical edges}, since they are the images of
vertical edges of $X_{m,n}$. They are directed, consistently with
$e$. 

\begin{proposition}\label{fiveconditions}
Suppose $m\not= n$ and let $G$ be a torsion-free uniform lattice in
$\Aut(X_{m,n})$. Let $k = \gcd(m,n)$, $m' = m/k$, and
$n'=n/k$. Let $(A,\lambda)$ be a directed labeled graph such that
$X_{m,n}/G$ is homeomorphic to $Z_{(A,\lambda)}$, with 
associated length function $\ell \co V(A) \sqcup E(A) \to \N$ and
directed structure induced from $T_{m,n}$. Then
\begin{enumerate}
\item\label{j1} 
  for every $v \in V(A)$,
  \[ \sum_{e \in E^+_0(v)} \abs{\lambda(e)} \ = \ n \ \text{
    and } \ \sum_{e \in E^-_0(v)} \abs{\lambda(e)} \ = \ m \] 
\item\label{j2} 
  for every $e \in E^+(A)$,
  \begin{align*}
  \ell(\partial_0(e)) \abs{\lambda(e)} \ &= \ n' \ell(e), \\
  \ell(\partial_1(e)) \abs{\lambda(\overline{e})} \ &= \ m' \ell(e)
  \end{align*}
\item\label{j5} 
  for every $v \in V(A)$, let $k_0(v) = \gcd(\ell(v),
  n')$ and $k_1(v) = \gcd(\ell(v),m')$; then there exist partitions
  \[ E^+_0(v) \ = \ E^+_1 \sqcup \dotsm \sqcup E^+_{k_0(v)}, \quad
  E^-_0(v) \ = \ E^-_1 \sqcup \dotsm \sqcup E^-_{k_1(v)}\]
  such that the sums 
  $\sum_{e \in E^+_i} \abs{\lambda(e)}$ are equal for all $i$, 
  and the sums $\sum_{e \in E^-_j} \abs{\lambda(e)}$ are equal for all
  $j$. 
\item\label{j3} for every $v \in V(A)$,
  \[\sum_{e \in E^+_0(v)} \ell(e) \ = \ k \ell(v) \ = \ \sum_{e \in
    E^-_0(v)} \ell(e)\]
\item\label{j4} the directed graph $A$ is strongly connected. 
\end{enumerate}
\end{proposition}

\begin{proof}
Conclusion \eqref{j1} follows exactly as in Proposition
\ref{latticeprop}. Conclusion \eqref{j2} is evident from the cell
structure on the annulus $M_e \cup M_{\overline{e}}$, which is tiled
by $(m',n')$--cells. Its boundary curves have lengths $n'\ell(e)$ and
$m'\ell(e)$, and they wrap $\abs{\lambda(e)}$ times and
$\abs{\lambda(\overline{e})}$ times respectively around
$C_{\partial_0(e)}$ and $C_{\partial_1(e)}$. 

For \eqref{j5}, identify the vertices of $C_v$ with the cyclic group
$\Z/\ell(v)\Z$ (in their natural cyclic ordering). The element $n' +
\ell(v)\Z$ generates the subgroup $k_0(v)\Z / \ell(v)\Z$. Given $e \in
E^+_0(v)$ the annulus $M_e \cup M_{\overline{e}}$ has $\ell(e)$
outgoing vertical edges incident to vertices of $C_v$, spaced $n'$ units
apart. They meet the vertices along a coset of $k_0(v)\Z / \ell(v)\Z$
in $\Z/\ell(v)\Z$, with $\ell(e) k_0(v)/\ell(v)$ outgoing edges
incident to each such vertex. 

For each $i = 1, \dotsc, k_0(v)$ let $E_i^+$ be the set of edges $e
\in E^+_0(v)$ such that the vertical edges of $M_e \cup
M_{\overline{e}}$ are joined to $C_v$ along the coset $i + (k_0(v)\Z /
\ell(v)\Z)$ in $\Z/\ell(v)\Z$. Now every vertex in the coset has 
\[ \sum_{e\in E^+_i} \ell(e) k_0(v) / \ell(v) \
  \overset{\eqref{j2}}{=} \ \sum_{e\in 
  E^+_i} \abs{\lambda(e)} k_0(v) / n' \ = \ (k_0(v)/n') \sum_{e\in
  E^+_i} \abs{\lambda(e)}\] 
outgoing vertical edges incident to it. In the universal cover
$X_{m,n}$, every vertex has the same number of outgoing vertical
edges; hence the same is true of $X_{m,n}/G$ and the sums $\sum_{e \in
  E^+_i} \abs{\lambda(e)}$ are the same for all $i$. The statement for
$E^-_0(v)$ is proved similarly. 

Conclusion \eqref{j3} follows from \eqref{j1} and \eqref{j2}:
\[
\sum_{e \in E^+_0(v)} \ell(e) \ \overset{\eqref{j2}}{=} \ \sum_{e \in E^+_0(v)}
\ell(v)\abs{\lambda(e)} / n' 
\ \overset{\eqref{j1}}{=} \ \ell(v) n / n' \ = \ k \ell(v)
\]
and similarly for the second sum.

For \eqref{j4}, define a new directed graph $A'$ from $A$ by
replacing each directed edge $e$ with $\ell(e)$ directed edges. By
\eqref{j3}, each vertex of $A'$ has equal numbers of incoming and
outgoing edges. It follows that $E^+(A')$ admits a partition into
directed cycles. Hence every directed edge is part of a directed
cycle. The same is then true of $A$. Now consider the decomposition of
$A$ into strongly connected components. If this decomposition is
non-trivial then there is a directed edge between two such components
that cannot be part of a directed cycle, which is a
contradiction. Hence $A$ is strongly connected. 
\end{proof}

\begin{theorem}\label{gbslatticechar}
Suppose $m\not= n$ and let $G$ be a torsion-free group. Then $G$ is
isomorphic to a uniform lattice in $\Aut(X_{m,n})$ if and only if
there exist a compact GBS structure $(A,\lambda)$ for $G$, a directed
graph structure $E(A) = E^+(A) \sqcup E^-(A)$, and a function $\ell
\co V(A) \cup E(A) \to \N$ satisfying $\ell(e) = \ell(\overline{e})$
for all $e \in E(A)$ such that conclusions \eqref{j1}, \eqref{j2}, and
\eqref{j5} of Proposition \ref{fiveconditions} hold. 
\end{theorem}

\begin{proof}
The ``only if'' direction holds by Proposition
\ref{fiveconditions}. For the converse let $(A,\lambda)$, the directed
graph structure, and $\ell$ be given. Let $k = \gcd(m,n)$, $m' = m/k$,
and $n' = n/k$. We put a cell stucture on $Z_{(A,\lambda)}$ as
follows. For each $v\in V(A)$ let $C_v$ have $\ell(v)$ vertices and
$\ell(v)$ edges. Choose an identification of the cyclically ordered
vertices with $\Z/\ell(v)\Z$. For each $e \in E^+(A)$ let $A_e$ be an
annulus tiled by $\ell(e)$ $(m',n')$-cells, so that the initial end is
a circle of length $n'\ell(e)$ and the opposite end has length 
$m'\ell(e)$. The $\ell(e)$ edges joining the two ends are called
vertical edges, and are directed from the initial end to the other
end. 

For each $v$ and $e \in E^+_0(v)$ there are $k_0(v)$ ways,
combinatorially, to attach the initial end of $A_e$ to $C_v$ by a
degree $\lambda(e)$ map, corresponding to the coset of
$k_0(v)\Z/\ell(v)\Z$ in $\Z/\ell(v)\Z$ met by the vertical edges of
$A_e$. (The existence of these combinatorial attaching maps follows
from \eqref{j2}.) Using the partition $E^+_0(v) = E^+_1 \sqcup \dotsm
\sqcup E^+_{k_0(v)}$, attach each annulus $A_e$ along the coset $i +
(k_0(v)\Z/\ell(v)\Z)$ such that $e \in E^+_i$. Attach the other ends of
each annulus similarly along the cosets of $k_1(v)\Z/\ell(v)\Z$. The
resulting space is homeomorphic to $Z_{(A,\lambda)}$. 

By \eqref{j1} we have $T_{(A,\lambda)} \cong T_{m,n}$ as directed
graphs. By construction, each strip in the universal cover of
$Z_{(A,\lambda)}$ is an $(m',n')$--strip. Finally, by \eqref{j5}, each
vertex of $Z_{(A,\lambda)}$ has equal numbers of outgoing (respectively,
incoming) vertical edges. It follows that the strips in the universal
cover are joined to the branching lines in the manner described in
Section \ref{combinatorialsect}. Hence the universal cover of
$Z_{(A,\lambda)}$ is isomorphic to $X_{m,n}$ as a CW complex, and $G$
acts freely and cocompactly on $X_{m,n}$. 
\end{proof}

\subsection{Discreteness}
One clear situation in which $\Aut(X_{m,n})$ cannot contain
incommensurable lattices is when it is discrete. When this occurs,
every lattice has finite index in $\Aut(X_{m,n})$, by
\cite[1.7]{basslubotzky}. 

\begin{theorem}\label{discrete} 
For any $m,n\geq 1$ the group $\Aut(X_{m,n})$ is discrete if and only
if\/ $\gcd(m,n) = 1$. When this occurs, there is a short exact
sequence 
\[ 1 \to BS(m,n) \to \Aut(X_{m,n}) \overset{\chi}{\to} \{\pm 1\} \to 1.\]
\end{theorem}

\begin{proof}
Suppose $\gcd(m,n) = 1$ and let $K$ be the kernel of $\chi \co
\Aut(X_{m,n}) \to \{\pm 1\}$. We claim that $K$ acts freely on
the vertices of $X_{m,n}$. If $g\in K$ fixes a vertex $x \in X_{m,n}$
then it also fixes pointwise the branching line $L$ containing
$x$. Now note from Section \ref{combinatorialsect} that the strips
above this branching line are all attached along different cosets $i +
n\Z$, and the same is true of the strips below $L$. Hence $g$ cannot
permute these strips and it acts trivially on them. By similar 
reasoning, $g$ acts trivially on the radius $r$ neighborhood of $L$,
for every $r$, so $g = 1$. 

The lattice $BS(m,n)$ is a subgroup of $K$ acting transitively on
the vertices of $X_{m,n}$. Since $K$ acts freely on these vertices, we
must have $K = BS(m,n)$. 

Now suppose that $\gcd(m,n) = k \not= 1$. Let $K = \ker(\chi)$ and
consider a vertex stabilizer $K_x$. Choose any directed ray
$\alpha$ in $T_{m,n}$ with initial vertex $\pi(x)$. Let $L_x$ be the
branching line through $x$ and for any vertex $v \in \alpha$ let $L_v$
be the branching line $\pi^{-1}(v)$. There are $k$ strips attached
above $L_v$ along each coset $i + (n/k)\Z$. In particular there are
elements $g \in K_x$ which fix $L_x$ and $L_v$ pointwise, but permute
the strips above $L_v$ nontrivially. Since $v$ is arbitrarily far
from $x$, $K_x$ is infinite and $K$ is non-discrete. (With a little
more care one may express $K_x$ explicitly as an inverse limit of
products of permutation groups.) 
\end{proof}

\section{A commensurability invariant}\label{invariantsec}

Consider a non-elementary GBS group $G$ and a non-trivial elliptic
subgroup $V < G$. Recall that because $G$ is non-elementary, the
elliptic subgroups of $G$ are well defined, as is the modular
homomorphism $q \co G \to \Q^{\times}$. 

We define the \emph{$V$\!-depth} of an element $g\in G$ to be $D_V(g)
= [V: V \cap V^g]$ . Next we define the \emph{depth profile}:
\[\mathcal{D}(G,V) \ = \ \{ D_V(g) \mid g \text{ is hyperbolic
    and } q(g) = \pm 1\} \ \subset \ \N. \]
The depth profile is a commensurability invariant in the following
sense.
\begin{theorem}\label{invariant}
Let $G$ be a non-elementary GBS group and $V<G$ a non-trivial elliptic
subgroup.
\begin{enumerate}
\item\label{t1} If\/ $G'<G$  is a subgroup of finite index
  and $V \subset G'$ then \[\mathcal{D}(G,V) \ = \ \mathcal{D}(G',V).\]
\item\label{t2} If\/ $V' < V$ with $[V:V'] = r$ then
  \[\mathcal{D}(G,V') \ = \ \{ \, n/\gcd(r,n) \mid n \in
  \mathcal{D}(G,V)\, \}.\] 
\end{enumerate}
\end{theorem}

Once $\mathcal{D}(G,V)$ is known, using \eqref{t2} one can compute 
depth profiles for every finite index subgroup of $G$, and
hence for every GBS group commensurable with $G$, by \eqref{t1}. 
Alternatively, one may define an equivalence relation on the set of
subsets of $\N$, by declaring $S \subset \N$ equivalent to $\{
\, n/\gcd(r,n) \mid n \in S \, \}$ for each $r \in \N$ and taking the
symmetric and transitive closure. Then the equivalence class of
$\mathcal{D}(G,V)$ is a true commensurability invariant of $G$. 

Given $S \subset \N$ and $r\in \N$ let $S/r$ denote the set $\{
\, n/\gcd(r,n) \mid n \in S \, \} \subset \N$.  

\begin{proposition}
Two subsets $S, S' \subset \N$ are equivalent if and only if there
exist $r, r' \in \N$ such that $S/r = S'/r'$. 
\end{proposition}

\begin{proof}
Let $\divideby_r \co \N \to \N$ denote the function $n \mapsto
n/\gcd(r,n)$. One easily verifies that $\divideby_s \circ
\divideby_r = \divideby_{rs}$, and therefore $(S/r)/s$ = $S/(rs)$
for all $S \subset \N$ and $r,s \in \N$. 

Let $\sim$ denote the smallest equivalence relation on
$\mathcal{P}(\N)$ such that $S \sim S/r$ for all $r\in \N$. Define $S
\approx S'$ if there exist $r, r'\in \N$ such that $S/r =
S'/r'$. Clearly, $S \approx S'$ implies $S \sim S'$. For the converse
it suffices to show that $\approx$ is an equivalence relation,
i.e. that it is transitive. If $S \approx S'$ and $S' \approx S''$
then there exist $r, r', s', s'' \in \N$ such that $S/r = S'/r'$ and
$S'/s' = S''/s''$. Then 
\[S/(rs') = (S/r)/s' = (S'/r')/s' = S' / (r's') = (S' /
s')/r' = (S'' / s'')/r' = S'' / (s'' r'),\] 
so indeed $S \approx S''$. 
\end{proof}

\begin{example}\label{bsexample}
Let $G = BS(k, kn)$ with $k,{n} > 1$. There is an index $k$ normal
subgroup $G' < G$ isomorphic to $\bigvee_{i=1}^k BS(1,n)$, with vertex
group $V$. Since every edge in this decomposition of $G'$ has indices
$1$ and $n$, every element of modulus $\pm 1$ has $V$-depth a power
of ${n}$. All powers are realized, and 
\[ \mathcal{D}(G,V) \ = \ \mathcal{D}(G',V) \ = \ \{ \, {n}^i \mid
  i \in \N \cup \{0\}\,\}.\]  
See Proposition \ref{wedgecomputation} ahead for a more detailed proof
(since by slide moves we can write $G' \cong BS(1,n) \vee
\bigvee_{i=1}^{k-1} BS(1,1)$). It is important that $k > 1$; otherwise 
there are no hyperbolic elements of modulus $\pm 1$ and the depth
profile of $G'$ is empty. 

Changing the elliptic subgroup to $V' < V$ with $[V,V']=r$, Theorem
\ref{invariant}\eqref{t2} gives
\[ \mathcal{D}(G, V') \ = \ \mathcal{D}(G,V)/r\ = \ \{ \,
   {n}^i/\gcd(r, {n}^i) \mid i 
  \in \N \cup \{0\}\,\}.\]
This set (for any $r$) has the property that, with finitely many
exceptions, any two successive elements have ratio ${n}$. Hence the
modulus ${n}$ is an invariant of the equivalence class of depth
profiles of $BS(k,kn)$. This yields a new proof of the most difficult
case of \cite[Theorem 1.1]{CRKZ}: 
\end{example}

\begin{corollary}[Lemma 7.2 of \cite{CRKZ}]\label{crkz} 
The groups $BS(k_1, k_1 n_1)$ and $BS(k_2, k_2n_2)$ with $k_i,
{n_i}>1$ are commensurable only if\/ ${n_1} = {n_2}$. 
\end{corollary}

Theorem \ref{invariant} will follow quickly from the next lemma. 

\begin{lemma}\label{invariantlemma} 
Let $G$ be a non-elementary GBS group and $V < G$ a non-trivial 
elliptic subgroup. Suppose $g$ is hyperbolic and $q(g) = \pm 1$. 
\begin{enumerate}
\item\label{l1} $D_V(g) = D_V(g^k)$ for all $k \geq 1$. 
\item\label{l2} If\/ $V' < V$ with $[V:V']=r$ then $D_{V'}(g) = D_V(g) /
  \gcd(r,D_V(g))$. 
\end{enumerate}
\end{lemma}

\begin{proof}
First we prove \eqref{l2}. Let $d =
D_V(g)$ and let $V = \langle x \rangle$, $V^g = \langle y
\rangle$. Then $V \cap V^g$ is the subgroup $\langle
x^d\rangle$. Since $q(g) = \pm 1$ we have $[V^g: V \cap V^g] = d$ and
so $V \cap V^g = \langle y^d\rangle$ also. We wish to identify the
subgroup $V' \cap V'^g$, given that $V' = \langle x^r\rangle$ and
$V'^g = \langle y^r \rangle$. 
Since $\langle x^r\rangle \cap \langle y^r\rangle \subset \langle x
\rangle \cap \langle y \rangle = \langle x^d \rangle = \langle y^d
\rangle$, it follows that 
\begin{align}
\langle x^r \rangle \cap \langle y^r\rangle \ 
&= \ \langle x^r \rangle \cap \langle x^d \rangle \cap \langle y^r
\rangle \cap \langle y^d \rangle \notag\\
&= \ \langle x^{rd/\gcd(r,d)} \rangle \cap \langle
  y^{rd/\gcd(r,d)}\rangle \notag\\ 
&= \ \langle x^{rd/\gcd(r,d)} \rangle \ = \  \langle
  y^{rd/\gcd(r,d)}\rangle. \label{line3} 
\end{align}
The equalities of line \eqref{line3} hold because both $\langle
x^{rd/\gcd(r,d)} \rangle$ and $\langle y^{rd/\gcd(r,d)}\rangle$ are
the unique subgroup of $\langle x^d \rangle = \langle y^d \rangle$ of
index $r/\gcd(r,d)$. Since $V' \cap V'^g = \langle x^{rd/\gcd(r,d)}
\rangle$ and $V' = \langle x^r\rangle$, we conclude that $D_{V'}(g) =
d/\gcd(r,d)$. 

For \eqref{l1}, let $X$ be a GBS tree for $G$. Then $V< G_x$ for some
vertex $x\in X$ and $[G_x:V] < \infty$. From \eqref{l2} we can say
that $D_{G_x}(g) = D_{G_x}(g^k)$ implies $D_V(g) = D_V(g^k)$, and thus
it suffices to establish \eqref{l1} when $V$ is a vertex stabilizer. 

Let $V = G_x$ for some vertex $x \in X$. Let $X_g \subset X$ be the
axis of $g$ and $y\in X_g$ the vertex
closest to $x$. Define the segments $\tau = [x,y]$, $\sigma_1 =
[y,gy]$, $\sigma_i = g^{i-1}\sigma_1$ (for $1 < i \leq k$) and $\sigma
= \sigma_1 \dotsm \sigma_k = [y, g^k y]$. Then $[x,gx] = \tau \cdot
\sigma_1 \cdot g \overline{\tau}$ and $[x, g^k x] = \tau \cdot \sigma
\cdot g^k \overline{\tau}$. 

Let $d = i(\sigma_1)$. Since $q(g) = \pm 1$ we have
$i(\overline{\sigma}_1) = d$ as well, and so $i(\sigma_i)=
i(\overline{\sigma}_i) = d$ for all $i$. By Lemma \ref{segmentindex}
it follows that $i(\sigma) = i(\overline{\sigma}) = d$. Now both
$[x,gx]$ and $[x,g^kx]$ are expressed as concatenations of three
segments, with matching indices along each. Hence (by Lemma
\ref{segmentindex}) $i([x,gx]) = i([x,g^kx])$. Since $D_V(g) =
i([x,gx])$ and $D_V(g^k) = [x, g^k x])$, these $V$-depths are equal. 
%
%If $g$ is elliptic let $y$ be a vertex closest
%to $x$ that is fixed by $g$. Let $\tau = [x,y]$. Then $[x,gx] = \tau
%\cdot g \overline{\tau}$ and $[x,g^kx] = \tau \cdot g^k
%\overline{\tau}$. These two concatenations have matching indices along
%them, so $D_V(g) = i([x,gx]) = i([x,g^kx]) = D_V(g^k)$ as in the
%previous case. 
\end{proof}

\begin{proof}[Proof of Theorem \ref{invariant}.] For \eqref{t1}, it is
immediate that $\mathcal{D}(G',V) \subset \mathcal{D}(G,V)$. For the
reverse inclusion, given $D_V(g) \in \mathcal{D}(G,V)$, choose $k$
such that $g^k \in G'$. Then $D_V(g) \in \mathcal{D}(G',V)$ by 
Lemma \ref{invariantlemma}\eqref{l1}. Conclusion \eqref{t2} follows
directly from Lemma \ref{invariantlemma}\eqref{l2}. 
\end{proof}

\begin{definition}
Recall from \eqref{modulus} that if $\sigma = [x, gx]$ then
$\abs{q(g)} = i(\sigma) / i(\overline{\sigma})$. Thus, for any segment
$\sigma$, we will call $\sigma$ \emph{unimodular} if $i(\sigma) =
i(\overline{\sigma})$. 
\end{definition}

%The next result gives a description of the depth profile purely in
%terms of indices in the tree. 
The next result gives a useful description of the depth profile. 

\begin{proposition}\label{index-depth} 
Suppose $V = G_v$ for some vertex $v$. 
Define the set
\[ \mathcal{I}(v) \ = \ \{i(\sigma) \mid \sigma \text{ is a
    non-trivial unimodular segment with endpoints in } Gv\}.\]
Then \[\mathcal{D}(G,V) \ \subseteq \ \mathcal{I}(v) \ \subseteq
\ \mathcal{D}(G,V) \cup \{1\}.\]
\end{proposition}

In particular, if $1 \in \mathcal{D}(G,V)$ then $\mathcal{D}(G,V) =
\mathcal{I}(v)$. An example where these two sets differ is $BS(1,n)$
with its standard tree as in Example \ref{bsexample}. The depth
profile is empty but there exist unimodular segments of
index $1$, namely, all segments $[x,gx]$ with $g$ elliptic. 

\begin{proof}
The first inclusion is immediate from the definitions. If $g$ is
hyperbolic with $q(g)  =\pm 1$ then $D_V(g) = [V : V\cap V^g] =
i([v,gv])$ and $[v,gv]$ is non-trivial and unimodular. For the second
inclusion, suppose $\sigma$ is a non-trivial unimodular segment with
endpoints in $Gv$. Applying a translation in $G$ (which does not
change $i(\sigma)$), we may assume that $\partial_0(\sigma) = v$. If
$i(\sigma) >1$ then by Lemma \ref{hypelement} we have that $\sigma = 
[v,hv]$ with $h$ hyperbolic, and so $i(\sigma) = D_V(h) \in
\mathcal{D}(G,V)$. Otherwise, $i(\sigma) = 1$. 
\end{proof}

\begin{remark}
The definition of the depth profile has some similarities with the
\emph{scale function} of a totally disconnected locally compact group,
defined by Willis \cite{willis}. For the scale to be defined in our
setting one must form the completion of the GBS group $G$ with respect
to the elliptic subgroups to obtain a 
totally disconnected locally compact group $\overline{G}$. In the case of $G =
BS(m,n)$ %with $m \not= \pm n$ 
this was carried out by Elder and Willis in
\cite{elderwillis}. They also gave a complete computation of
the scale function. The scale of an element of $\overline{G}$ turns out to depend
only on its $t$--exponent. Write $m=ab$, $n=bc$ where $b =
\gcd(m,n)$ and let $\epsilon\co \overline{G} \to \Z$ be the $t$--exponent
function, defined by continuous extension from $G$. Then the scale
of $g$ is %$\abs{a}^{\epsilon(g)}$ if $\epsilon(g) \geq 0$ or
%$\abs{c}^{-\epsilon(g)}$ if $\epsilon(g) < 0$.  
either $\abs{a}^{\epsilon(g)}$ or 
  $\abs{c}^{-\epsilon(g)}$, whichever one has non-negative exponent. 

Recall that the depth profile uses only the hyperbolic elements of
modulus $\pm 1$, which in this case are those with $t$--exponent zero
(and scale $1$). 
Consider the standard HNN structure for $G$ and let $V$ be the
index--$b$ subgroup of the 
vertex group. (This choice gives the cleanest result.) Let us assume that $\abs{m}, \abs{n}
> 1$, so that hyperbolic elements of modulus $\pm 1$ exist. 
A lengthy computation using
Proposition \ref{index-depth} and Lemma \ref{segmentindex}/Remark
\ref{shortindex} shows that 
\[\mathcal{D}(G,V) \ = \  \{ \, \abs{a}^i \abs{c}^j \mid i,j \in
  \N \cup \{0\}, \ i+j>0 \,\}.\]
It appears that the depth profile and the scale function are detecting 
quite different information. 
\end{remark}

\section{The main examples}\label{examplessec}

Our main examples will be lattices in $\Aut(X_{k, k n})$ for $k, n >
1$. We fix some notation: define $a,b,c$ such that $b = \gcd(k,n)$, $k
= ab$, and $n =  bc$.

\subsection{The lattice $G_1$}
This group is the index $2$ subgroup
of $BS(k, kn)$ generated by $a$, $tat^{-1}$, and $t^2$. It has a
labeled graph description as shown below, with two vertices and two
edges. 
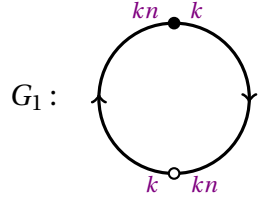
\[G_1: \quad 
\vcenter{\hbox{
\begin{tikzpicture} 
\scriptsize

\draw[very thick] (2,2) ellipse (1 and 1);

\filldraw[fill=white,thick] (2,1) circle (.7mm);
\filldraw[fill=black,thick] (2,3) circle (.7mm);

\draw[very thick,->] (3,1.96) -- (3,1.95);
\draw[very thick,->] (1,2.04) -- (1,2.05);

\draw[violet] (1.9,2.97) node[anchor=south east] {$kn$};
\draw[violet] (2.1,2.97) node[anchor=south west] {$k$};
\draw[violet] (2.4,.75) node[anchor=base] {$kn$};
\draw[violet] (1.7,.75) node[anchor=base] {$k$};
\end{tikzpicture}
\hspace*{1.02cm}
}}\]

Depth profiles of $BS(k,kn)$ were computed in Example
\ref{bsexample} (see also Proposition \ref{wedgecomputation}
below). Hence, for a suitable choice of elliptic subgroup 
$V_1 < G_1$, the depth profile is
\[ \mathcal{D}(G_1, V_1) \ = \ \{\, n^i \mid i \in \N \cup
  \{0\}\,\}.\]

\subsection{The lattice $G_2$}
This group is defined by the directed labeled graph $(B,\mu)$
below. It is bipartite with two vertices $v$ (white) and $u$ (black),
and $k+1$ directed edges. The edges $e_1, \dotsc, e_k$ are directed
from $u$ to $v$ and the edge $e_0$ is directed from $v$ to $u$. We
have $\mu(e_0) = k$, $\mu(\overline{e}_0) = kn$ and $\mu(e_i) = 1$,
$\mu(\overline{e}_i) = n$ for $i \not= 0$. 
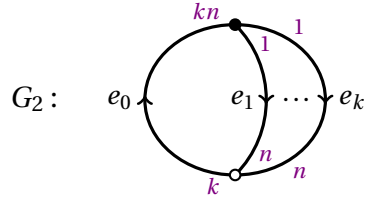
\[G_2: \quad 
\vcenter{\hbox{
\begin{tikzpicture}
\small

\draw[very thick] (2,2) ellipse (1.2 and 1);
\draw[very thick] (2,1) arc(-45:45:1.4142126);
\filldraw[fill=white,thick] (2,1) circle (.7mm);
\filldraw[fill=black,thick] (2,3) circle (.7mm);

\draw[very thick,->] (3.2,1.96) -- (3.2,1.95);
\draw[very thick,->] (2.4142126,1.96) -- (2.4142126,1.95);
\draw[very thick,->] (0.8,2.04) -- (0.8,2.05);
\draw[color=black] (2.83,1.97) node {$\dotsm$};

\draw (2.1,2) node {$e_1$};
\draw (3.25,2) node[anchor=west] {$e_k$};
\draw (0.75,2) node[anchor=east] {$e_0$};

\scriptsize
\draw[violet] (1.9,2.97) node[anchor=south east] {$kn$};

\draw[violet] (2.85,2.97) node {$1$};

\draw[violet] (2.4,2.75) node {$1$};
\draw[violet] (2.4,1.25) node {$n$};
\draw[violet] (2.85,1.03) node {$n$};

\draw[violet] (1.7,.75) node[anchor=base] {$k$};
\end{tikzpicture}
\hspace*{1.02cm}
}}\]

Both $G_1$ and $G_2$ are lattices in $\Aut(X_{k,kn})$ by Proposition
\ref{latticeprop}.

\subsection{The subgroup $H_2$} 
We will define a finite index subgroup $H_2 < G_2$
by constructing an admissible branched covering $(A, \lambda)$ of 
$(B,\mu)$. This subgroup will aid us in computing a depth profile
for $G_2$.

The graph $A$ has one vertex $v_1$ above $v$ and $b$ vertices $u_1,
\dotsc, u_b$ above $u$. For each $i$ there are $k$ directed edges from
$u_i$ to $v_1$, mapping to $e_1, \dotsc, e_k$ respectively, and $a$
directed edges from $v_1$ to 
$u_i$, mapping to $e_0$. If $e$ is a directed edge above $e_0$ then
its labels are $\lambda(e) = 1$ and $\lambda(\overline{e}) = b^2
c$. If $e$ is a directed edge above $e_i$ ($i \geq 1$) then its labels
are $\lambda(e) = 1$ and $\lambda(\overline{e}) = c$.

Finally, the
degree function of the branched covering is given by $d(v_1) = k$,
$d(u_i) = a$, $d(e) = 1$ for every edge $e$ above $e_0$, and $d(e) =
a$ for every edge $e$ above $e_i$ ($i \geq 1$). See Figure
\ref{h2graph}. One may verify that the conditions of Definition
\ref{admissibledef} are met. 

\begin{figure}[!ht]
  $\vcenter{\hbox{ 
\begin{tikzpicture} [scale=1.2]

\begin{scope}[xshift=-0.3cm]
\small
\begin{scope}[rotate around={-30:(2,3)}]
\draw[very thick] (0,3) -- (2,3);
\filldraw[fill=black,thick] (0,3) circle (.5mm);
\draw (0,3) node[anchor=east] {$u_b \ $};
\draw[very thick,->] (1.04,3) -- (1.05,3);
\end{scope}

\begin{scope}[rotate around={30:(2,3)}]
\draw[very thick] (0,3) -- (2,3);
\filldraw[fill=black,thick] (0,3) circle (.5mm);
\draw (0,3) node[anchor=east] {$u_1 \ $};
\draw[very thick,->] (1.04,3) -- (1.05,3);
\end{scope}
\filldraw[fill=white,thick] (2,3) circle (.5mm);
\draw (2,3) node[anchor=west] {$v_1$};
\draw[color=black] (0.268,3.09) node {$\vdots$};

\draw[thick,color=mygreen,->] (0.268,1.6) -- (0.268,1);
\draw[thick,color=mygreen,->] (1.174,1.9) -- (1.174,1.1);
\draw[thick,color=mygreen,->] (2,2.2) -- (2,1.2);

\draw[very thick] (0.268,.5) -- (2,.5);
\draw[very thick,->] (1.174,.5) -- (1.184,.5);
\filldraw[fill=black,thick] (0.268,.5) circle (.5mm);
\draw (.268,.5) node[anchor=east] {$u \ $};
\filldraw[fill=white,thick] (2,.5) circle (.5mm);
\draw (2,.5) node[anchor=west] {$\, v$};
\draw (1.18,.4) node[anchor=north] {$e_i$};

\scriptsize
\draw[violet] (0.5,4.09) node {$1$};
\draw[violet] (1.85,3.3) node {$c$};
\draw[violet] (0.5,1.91) node {$1$};
\draw[violet] (1.85,2.7) node {$c$};

\draw[color=mygreen] (0.268,1.3) node[anchor=east] {$a$};
\draw[color=mygreen] (1.174,1.5) node[anchor=east] {$a$};
\draw[color=mygreen] (2,1.72) node[anchor=east] {$k$};

\draw[violet] (0.5,.53) node[anchor=south] {$1$};
\draw[violet] (1.81,.53) node[anchor=south] {$n$};
\end{scope}

\small
\begin{scope}[rotate around={30:(3,3)}]
\draw[very thick] (3,3) arc(120:60:2);
\draw[very thick] (3,3) arc(240:300:2);
\filldraw[fill=black,thick] (5,3) circle (.5mm);
\draw[color=black] (4,3.08) node[transform shape,scale=.7] {$\vdots$};
\draw (5,3) node[anchor=west] {$ \ u_b$};
\draw[very thick,->] (4.04,3.2679) -- (4.05,3.2679);
\draw[very thick,->] (4.04,2.7321) -- (4.05,2.7321);
\end{scope}

\begin{scope}[rotate around={-30:(3,3)}]
\draw[very thick] (3,3) arc(120:60:2);
\draw[very thick] (3,3) arc(240:300:2);
\filldraw[fill=black,thick] (5,3) circle (.5mm);
\draw[color=black] (4,3.08) node[transform shape,scale=.7] {$\vdots$};
\draw (5,3) node[anchor=west] {$ \ u_1$};
\draw[very thick,->] (4.04,3.2679) -- (4.05,3.2679);
\draw[very thick,->] (4.04,2.7321) -- (4.05,2.7321);
\end{scope}
\filldraw[fill=white,thick] (3,3) circle (.5mm);
\draw (3,3) node[anchor=east] {$v_1$};
\draw[color=black] (4.732,3.09) node {$\vdots$};

\draw[thick,color=mygreen,->] (4.732,1.6) -- (4.732,1);
\draw[thick,color=mygreen,->] (3.906,1.8) -- (3.906,1);
\draw[thick,color=mygreen,->] (3,2.2) -- (3,1.2);

\draw[very thick] (3,.5) -- (4.732,.5);
\draw[very thick,->] (3.906,.5) -- (3.916,.5);
\filldraw[fill=white,thick] (3,.5) circle (.5mm);
\draw (3,.5) node[anchor=east] {$v \,$};
\filldraw[fill=black,thick] (4.732,.5) circle (.5mm);
\draw (4.732,.5) node[anchor=west] {$ \ u$};
\draw (3.92,.4) node[anchor=north] {$e_0$};

\scriptsize
\draw[color=black] (4.05,3.6) node {$a$};
\draw[color=black] (4.05,2.4) node {$a$};

\draw[violet] (3.01,3.35) node {$1$};
\draw[violet] (3.41,3.23) node {$1$};
\draw[violet] (4.38,4.2) node {$b^2 c$};
\draw[violet] (4.83,3.67) node {$b^2 c$};
\draw[violet] (4.83,2.4) node {$b^2 c$};
\draw[violet] (4.38,1.83) node {$b^2 c$};
\draw[violet] (3.02,2.65) node {$1$};
\draw[violet] (3.41,2.77) node {$1$};

\draw[color=mygreen] (3,1.72) node[anchor=east] {$k$};
\draw[color=mygreen] (3.906,1.42) node[anchor=east] {$1$};
\draw[color=mygreen] (4.732,1.3) node[anchor=west] {$a$};

\draw[violet] (3.19,.53) node[anchor=south] {$k$};
\draw[violet] (4.45,.53) node[anchor=south] {$kn$};
\end{tikzpicture}
}} \quad 
  \quad H_2\colon \ \  {\displaystyle \bigvee_{i=1}^b} \ \ 
  \vcenter{\hbox{
{
\begin{tikzpicture}[scale=1.2]
\small

\draw[very thick] (2,2) ellipse (1.2 and 1);
\draw[very thick] (2,1) arc(-45:45:1.4142126);
\draw[very thick] (2,1) arc(225:135:1.4142126);
\filldraw[fill=white,thick] (2,1) circle (.5mm);
\filldraw[fill=black,thick] (2,3) circle (.5mm);

\draw[very thick,->] (3.2,1.96) -- (3.2,1.95);
\draw[very thick,->] (2.4142126,1.96) -- (2.4142126,1.95);
\draw[very thick,->] (1.5857874,2.04) -- (1.5857874,2.05);
\draw[very thick,->] (0.8,2.04) -- (0.8,2.05);
\draw[color=black] (2.83,1.97) node {$\dotsm$};
\draw[color=black] (1.21,1.97) node {$\dotsm$};

\draw (2,0.9) node[anchor=north] {$v_1$};
\draw (2,3.1) node[anchor=south] {$u_i$};

\scriptsize
\draw[color=black] (1.1928937,2.03) node[anchor=south] {$a$};
\draw[color=black] (2.8071063,2.03) node[anchor=south] {$k$};

%\draw[violet] (1.8,2.95) node[anchor=south east] {$b^2 c$};
\draw[violet] (0.96,2.9) node {$b^2 c$};

%\draw[violet] (2.3,2.95) node[anchor=south west] {$1$};
\draw[violet] (2.89,2.9) node {$1$};

\draw[violet] (2.43,2.68) node {$1$};
\draw[violet] (1.44,2.68) node {$b^2 c$};
\draw[violet] (2.38,1.25) node {$c$};
\draw[violet] (1.63,1.25) node {$1$};
%\draw[violet] (2.4,.82) node[anchor=base] {$c$};
\draw[violet] (2.83,1.03) node[anchor=base] {$c$};

%\draw[violet] (1.57,.82) node[anchor=base] {$1$};
\draw[violet] (1.16,1.0) node[anchor=base] {$1$};
\end{tikzpicture}
}
}}$
\caption{The admissible branched covering defining $H_2$. Above each
  $e_i$ ($i \geq 1$)  there are $b$ edges as shown. Above $e_0$  there
  are $k = ab$ edges as shown. Overall, $(A,\lambda)$ is a wedge
  product of $b$ copies of a graph, joined at the vertex $v_1$. 
  }\label{h2graph}
\end{figure}
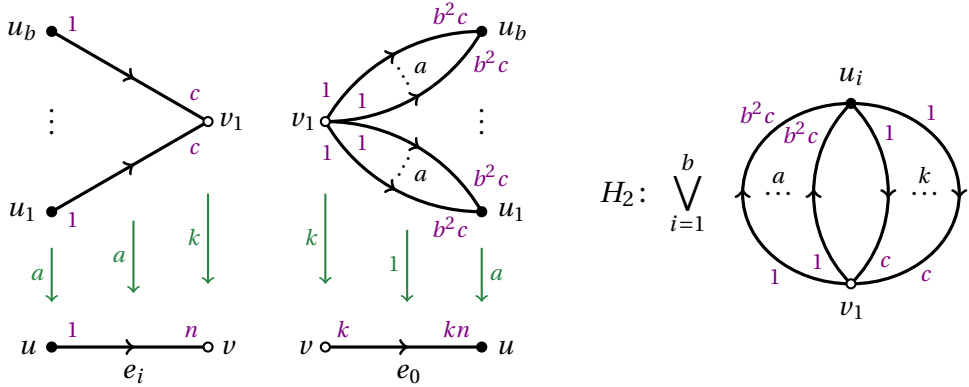
By collapsing each of the edges above $e_k$, and then performing $k-1$
slide moves, we find that 
\begin{align}
  H_2 \ &\cong \ \bigvee_k BS(1,n^2) \ \vee \bigvee_{b(k-1)} BS(c,c)\notag  \\
  &\cong \ BS(1,n^2) \ \vee \ \bigvee_{k-1} BS(1,1) \ \vee
  \bigvee_{b(k-1)} BS(c,c). \label{h2-final}
\end{align}
A depth profile of $H_2$ can now be computed using the next result.

\begin{proposition}\label{wedgecomputation}
  Suppose $G = BS(1,N) \vee \bigvee_{i=1}^r BS(n_i,n_i)$ for some $r
  \geq 1$ and suppose that $N>1$, each $n_i$ divides $N$, and the set
  $\{ n_1, \dotsc, n_r, N \, \}$ is closed under taking $\lcm$ and
  contains $1$. Let $V$ be
  the vertex group. Then
  \begin{equation}\label{formula}
    \mathcal{D}(G,V)  \ = \
    \{ \, N^in_j \mid i\in \N \cup\{0\}, j = 1, \dotsc, r \, \}.
    \end{equation}
\end{proposition}

\begin{proof}
Let $X$ be the Bass--Serre tree for the given GBS structure of
$G$. The subgroup $V$ is the stabilizer of a vertex $v$. Note that $G$
acts transitively on $V(X)$. Also, since $n_i =  1$  
for some $i$, the stable letter from $BS(n_i,n_i)$ is a
hyperbolic element with modulus $1$ and $V$-depth $1$. So
$\mathcal{D}(G,V) = \mathcal{I}(v)$, by Proposition
\ref{index-depth}. That is, $\mathcal{D}(G,V)$ is the set of indices
$i(\sigma)$ of non-trivial unimodular segments $\sigma$ in $X$. 

Call an edge $e \in E(X)$ \emph{ascending} if $i(e) = 1$ and
$i(\overline{e}) = N$, and \emph{descending} if $\overline{e}$ is
ascending. Note that every edge in $X$ is either ascending,
descending, or unimodular. Now every unimodular segment $\sigma$ of
length $> 1$ has one of the following forms: 
\begin{enumerate}
\item\label{c1} $\sigma_1 \sigma_2$ with $\sigma_1, \sigma_2$ unimodular 
\item\label{c2} $e_1 \tau e_2$ with $\tau$ unimodular and $e_1$
  ascending, $e_2$ descending 
\item\label{c3} $e_1 \tau e_2$ with $\tau$ unimodular and $e_1$
  descending, $e_2$ ascending.
\end{enumerate}

  Let $D$ denote the right hand side of \eqref{formula}. It is easily
verified that $D$ is closed under taking $\lcm$. We can now show that
every unimodular $\sigma$ has index in $D$ by induction on
length. Unimodular edges have index $n_j$ for some $j$, which are in
$D$. If $\sigma$ is of type \eqref{c1} then by Remark \ref{shortindex}
we have $i(\sigma) = \lcm(i(\sigma_1),i(\sigma_2)) \in D$. If $\sigma$
is of type \eqref{c2} with $i(\tau) = N^i n_j$, using Remark
\ref{shortindex} one finds that $i(\sigma) = N^{i-1}n_j$ if $i \geq 1$
and $i(\sigma) = 1$ if $i=0$. Hence $i(\sigma)\in D$. If $\sigma$ is
of type \eqref{c3} then $i(\sigma) = N i(\tau) \in D$.

Finally, consider a segment $\sigma = e_1 \dotsm e_i \tau e'_1
\dotsm e'_i$ where each $e_k$ is descending, $\tau$ is a unimodular
edge with index $n_j$, and each $e'_k$ is ascending. Then $i(\sigma) =
N^in_j$.  
\end{proof}

\begin{theorem}\label{mainthm}
If\/ $\gcd(k,n) \not= 1$ then the lattices $G_1, G_2 < \Aut(X_{k,kn})$
are not abstractly commensurable. 
\end{theorem}

\begin{proof}
Consider the GBS structure \eqref{h2-final} for $H_2$ and let $V_2$ be
its vertex group. By Theorem \ref{invariant} and Proposition
\ref{wedgecomputation} we have 
\[ \mathcal{D}(G_2, V_2) \ = \ \mathcal{D}(H_2,V_2) \ = \ \{\, n^{2i}
\mid i \in \N\cup \{0\}\,\} \ \cup \ \{\, n^{2i}c \mid i \in \N\cup
\{0\}\,\}. \]
We also have
\[ \mathcal{D}(G_1, V_1) \ = \ \{\, n^i \mid i \in \N \cup \{0\}\,\}\]
as mentioned earlier. 
Enumerating the elements of $\mathcal{D}(G_i,V_i)$ in order, notice
that each element divides the next one. Taking the ratios of
successive elements one obtains the sequences $(n, n, n, \dotsc )$ for
$i=1$ and $(c, n^2/c, c, n^2/c, \dotsc )$ for $i=2$. The tails of
these ratio sequences are unchanged when passing from
$\mathcal{D}(G_i, V_i)$ to $\mathcal{D}(G_i,V_i)/r$ for any $r \in
\N$, because the values of $\gcd(r,n^i)$, $\gcd(r,n^{2i})$, and
$\gcd(r,n^{2i}c)$ stabilize as $i \to \infty$, all to the same
number. (This number is $r$ with all of its prime factors not
dividing $n$ removed.) Since $c\not= n$, the 
two tails will never agree, and so the two depth profiles are
inequivalent. 
\end{proof}

\begin{remark}\label{mainthm-remark}
If $\gcd(k,n) =1$ then $c = n$ and the depth profiles of $G_1$ and
$G_2$  coincide. However, the depth profile is not failing us as an
invariant, as it turns out the groups are commensurable in this
case. The labeled graph on the right hand side of Figure \ref{h2graph}
is an admissible branched covering of the labeled graph for $G_1$,
since $b=1$ and $k=a$. 
\end{remark}

\begin{example}
Figure \ref{easycase} illustrates $G_1$ and $G_2$ in the simplest
case, when $X_{k,kn} = X_{2,4}$. In this case the lattices are
commensurable to $BS(2,4)$ and $BS(4,16)$ respectively, which are
incommensurable by \cite{CRKZ}. The vertical maps are admissible
branched coverings and the horizontal arrows are elementary
deformations. A similar phenomenon occurs whenever $n=k>1$: in
$\Aut(X_{k,k^2})$ there are incommensurable lattices commensurable to
$BS(k,k^2)$ and $BS(k^2,k^4)$ respectively. 
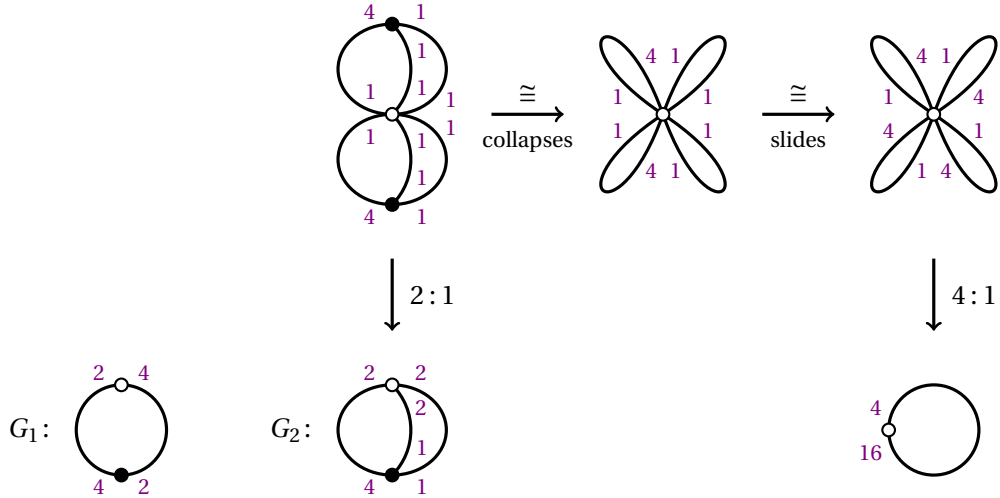
\begin{figure}[!ht]
\begin{tikzpicture}[scale=1.2]
\small

\draw[very thick] (2,1.5) circle (0.5);
\filldraw[fill=black,thick] (2,1) circle (.7mm);
\filldraw[fill=white,thick] (2,2) circle (.7mm);
\draw (1.4,1.5) node[anchor=east] {$G_1\colon$};

\draw[very thick] (5,1.5) ellipse (0.6 and 0.5);
\draw[very thick] (5,1) arc(-45:45:0.706063);
\filldraw[fill=black,thick] (5,1) circle (.7mm);
\filldraw[fill=white,thick] (5,2) circle (.7mm);
\draw (4.3,1.5) node[anchor=east] {$G_2\colon$};

\draw[very thick,->] (5,3.4) -- (5,2.6);

\draw[very thick] (5,4.5) ellipse (0.6 and 0.5);
\draw[very thick] (5,4) arc(-45:45:0.706063);
\draw[very thick] (5,5.5) ellipse (0.6 and 0.5);
\draw[very thick] (5,5) arc(-45:45:0.706063);
\filldraw[fill=black,thick] (5,4) circle (.7mm);
\filldraw[fill=white,thick] (5,5) circle (.7mm);
\filldraw[fill=black,thick] (5,6) circle (.7mm);

\draw[very thick,->] (6.1,5) -- (6.9,5);
\draw (6.5,5) node[anchor=south] {$\cong$};

\begin{scope}[rotate around={7:(8,5)}]
\draw[very thick] (8,5) .. controls (8.675,6.35) and (9.35,5.675) .. (8,5);
\draw[very thick] (8,5) .. controls (7.325,3.65) and (6.65,4.325) .. (8,5);
\end{scope}
\begin{scope}[rotate around={-7:(8,5)}]
\draw[very thick] (8,5) .. controls (9.35,4.325) and (8.675,3.65) .. (8,5);
\draw[very thick] (8,5) .. controls (6.65,5.675) and (7.325,6.35) .. (8,5);
\end{scope}
\filldraw[fill=white,thick] (8,5) circle (.7mm);

\draw[very thick,->] (9.1,5) -- (9.9,5);
\draw (9.5,5) node[anchor=south] {$\cong$};

\begin{scope}[rotate around={7:(11,5)}]
\draw[very thick] (11,5) .. controls (11.675,6.35) and (12.35,5.675) .. (11,5);
\draw[very thick] (11,5) .. controls (10.325,3.65) and (9.65,4.325) .. (11,5);
\end{scope}
\begin{scope}[rotate around={-7:(11,5)}]
\draw[very thick] (11,5) .. controls (12.35,4.325) and (11.675,3.65) .. (11,5);
\draw[very thick] (11,5) .. controls (9.65,5.675) and (10.325,6.35) .. (11,5);
\end{scope}
\filldraw[fill=white,thick] (11,5) circle (.7mm);

\draw[very thick,->] (11,3.4) -- (11,2.6);
\draw[very thick] (11,1.5) circle (0.5);
\filldraw[fill=white,thick] (10.5,1.5) circle (.7mm);

\draw (5.1,3) node[anchor=west] {$2:1$};
\draw (11.1,3) node[anchor=west] {$4:1$};

\scriptsize
\draw (6.5,4.9) node[anchor=north] {$\text{collapses}$};
\draw (9.5,4.9) node[anchor=north] {$\text{slides}$};

\draw[violet] (1.75,2.14) node {$2$};
\draw[violet] (2.25,2.14) node {$4$};
\draw[violet] (1.75,0.87) node {$4$};
\draw[violet] (2.25,0.87) node {$2$};

\draw[violet] (4.75,2.14) node {$2$};
\draw[violet] (4.75,0.87) node {$4$};
\draw[violet] (5.32,1.75) node {$2$};
\draw[violet] (5.33,1.3) node {$1$};
\draw[violet] (5.32,2.14) node {$2$};
\draw[violet] (5.33,0.87) node {$1$};

\draw[violet] (4.75,3.87) node {$4$};
\draw[violet] (5.33,3.87) node {$1$};
\draw[violet] (5.33,4.3) node {$1$};
\draw[violet] (5.33,4.71) node {$1$};
\draw[violet] (5.33,5.3) node {$1$};
\draw[violet] (5.33,5.71) node {$1$};

\draw[violet] (4.75,6.14) node {$4$};
\draw[violet] (5.32,6.14) node {$1$};
\draw[violet] (5.65,4.85) node {$1$};
\draw[violet] (5.65,5.16) node {$1$};
\draw[violet] (4.76,4.75) node {$1$};
\draw[violet] (4.76,5.25) node {$1$};

\draw[violet] (7.5,4.82) node {$1$};
\draw[violet] (7.5,5.19) node {$1$};
\draw[violet] (8.5,4.82) node {$1$};
\draw[violet] (8.5,5.19) node {$1$};
\draw[violet] (7.87,5.65) node {$4$};
\draw[violet] (8.14,5.65) node {$1$};
\draw[violet] (7.87,4.37) node {$4$};
\draw[violet] (8.14,4.37) node {$1$};

\draw[violet] (10.5,4.82) node {$4$};
\draw[violet] (10.5,5.19) node {$1$};
\draw[violet] (11.5,4.82) node {$1$};
\draw[violet] (11.5,5.19) node {$4$};
\draw[violet] (10.87,5.65) node {$4$};
\draw[violet] (11.14,5.65) node {$1$};
\draw[violet] (10.87,4.37) node {$1$};
\draw[violet] (11.14,4.37) node {$4$};

\draw[violet] (10.36,1.75) node {$4$};
\draw[violet] (10.3,1.25) node {$16$};
\end{tikzpicture}
\caption{Lattices $G_1, G_2$ in $\Aut(X_{2,4})$. $G_1$ is an index 2
  subgroup of $BS(2,4)$. $G_2$ contains an index 2 subgroup isomorphic
  to an index 4 subgroup of $BS(4,16)$ ($ \not\sim BS(2,4)$). For
  appropriate choices of elliptic subgroups $V_i < G_i$ the depth
  profiles are $\mathcal{D}(G_1,V_1) = \{\, 2^i \,\}$ and
  $\mathcal{D}(G_2,V_2) = \{\, 4^i \,\}$. 
  }\label{easycase}
\end{figure}
\end{example}

\begin{remark}\label{envelope}
Regarding Figure \ref{easycase}, notice that the finite-index subgroup
of $G_2$ is a lattice both in $\Aut(X_{2,4})$ and
$\Aut(X_{4,16})$, even though their ``standard'' lattices $BS(2,4)$
and $BS(4,16)$ are not commensurable. 
\end{remark}

\section{Further cases of $\Aut(X_{k,kn})$} \label{furthersec}

We return to the situation of Remark \ref{mainthm-remark},
when $\gcd(k,n) = 1$. 

\subsection{The lattice $G_3$}
Suppose that $p$ is a non-trivial divisor of $n$ (not necessarily
prime) such that $p < k$. Let $l = k-p$. The labeled graph below
defines a lattice $G_3 < \Aut(X_{k,kn})$ by Proposition
\ref{latticeprop}. It has vertices $v$ (white) and $u$ (black),
directed edges $e_1, \dotsc, e_k$ from $u$ to $v$, directed edges
$f_1, \dotsc, f_l$ from $v$ 
to $u$, and a directed edge $f_0$ from $v$ to $u$. The labels are
given by $\lambda(e_i) = 1$, $\lambda(\overline{e}_i) = n$,
$\lambda(f_0) = p$, $\lambda(\overline{f}_0) = pn$, and $\lambda(f_i)
= 1$, $\lambda(\overline{f}_i) = n$ ($i > 0$). 
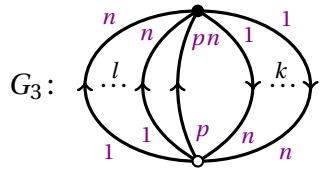
\[G_3\colon \ \ \vcenter{\hbox{
\begin{tikzpicture}
\small
\draw[very thick] (2,2) ellipse (1.5 and 1);
\draw[very thick] (2,1) arc(-45:45:2.5 and 1.4142126);
\draw[very thick] (2,1) arc(225:135:2.5 and 1.4142126);
\draw[very thick] (2,1) arc(210:150:2);
\filldraw[fill=white,thick] (2,1) circle (.7mm);
\filldraw[fill=black,thick] (2,3) circle (.7mm);

\draw[very thick,->] (3.5,1.96) -- (3.5,1.95);
\draw[very thick,->] (2.735,1.96) -- (2.735,1.95);
\draw[very thick,->] (0.5,2.04) -- (0.5,2.05);
\draw[very thick,->] (1.265,2.04) -- (1.265,2.05);
\draw[very thick,->] (1.732,2.04) -- (1.732,2.05);
\draw[color=black] (0.9,1.97) node {$\dotsm$};
\draw[color=black] (3.15,1.97) node {$\dotsm$};

\scriptsize
\draw[color=black] (0.9, 2.01) node[anchor=south] {$l$}; 
\draw[color=black] (3.1, 2.01) node[anchor=south] {$k$}; 
\draw[violet] (0.82,2.9) node {$n$};
\draw[violet] (1.32,2.68) node {$n$};
\draw[violet] (2.1,2.58) node {$pn$};
\draw[violet] (2.68,2.68) node {$1$};
\draw[violet] (3.18,2.9) node {$1$};
\draw[violet] (2.06,1.37) node {$p$};
\draw[violet] (2.66,1.32) node {$n$};
\draw[violet] (3.16,1.1) node {$n$};
\draw[violet] (1.32,1.35) node {$1$};
\draw[violet] (0.82,1.13) node {$1$};
\end{tikzpicture}
\hspace*{.85cm}
}}\]

\begin{theorem}\label{g3thm}
  Suppose $n$ has a non-trivial divisor $p\not= n$ such that $p <
  k$. Then the lattices $G_1, G_3 < \Aut(X_{k,kn})$ are not abstractly
  commensurable. 
\end{theorem}

\begin{proof}
  Collapsing $f_1$ and performing $k-1$ slide moves, we obtain
  \begin{align*}
    G_3 \ &\cong \ \bigvee_k BS(1,n^2) \ \vee \! 
      \bigvee_{k-p-1} BS(n,n) \vee BS(pn,pn) \\
      &\cong \ BS(1,n^2) \ \vee \ \bigvee_{k-1} BS(1,1) \ \vee \! 
      \bigvee_{k-p-1} BS(n,n) \ \vee \ BS(pn, pn).
  \end{align*}
  Let $V_3$ be the vertex group. Proposition \ref{wedgecomputation}
  provides the depth profile:
  \[ \mathcal{D}(G_3, V_3) \ = \
\begin{cases}
  \{\, n^{2i}\, \} \cup \{\, n^{2i+1}\, \} \cup \{\, n^{2i+1} p\, \} &
  \text{ if } p<k-1 \\
  \{\, n^{2i}\, \} \cup \{\, n^{2i+1} p\, \} & \text{ if }
  p=k-1. \\
  \end{cases} \]
In both cases, enumerating the elements in order, each element divides
the next. Hence there is a sequence of successive ratios that can be
compared to the sequence $(n, n, n, \dotsc)$ arising from
$\mathcal{D}(G_1,V_1)$. Exactly as in the proof of Theorem
\ref{mainthm}, the tails of the sequences, modulo shifting, are
invariants of the equivalence classes of depth profiles.

The sequences of ratios are
\[\begin{array}{ll}
\text{the $3$--periodic sequence } n, p, n/p, \ \dotsc & \text{ if }
  p<k-1 \\
  \text{the $2$--periodic sequence } np, n/p, \ \dotsc & \text{ if }
  p=k-1. \\
\end{array}\]
Neither of these sequences eventually agree with $(n, n, n, \dotsc)$,
so the two depth profiles are inequivalent in both cases. 
\end{proof}

\subsection{The lattice $G_4$} Suppose $n<k$ and $k \equiv 1 \mod
n$. Let $l = (k-1)/n$. The labeled graph below defines a lattice $G_4
< \Aut(X_{k,kn})$. It is bipartite with vertices $v$ (white) and $u$
(black), has directed edges $e_1, \dotsc, e_k$ from $u$ to $v$ with
labels $\lambda(e_i) = 1$, $\lambda(\overline{e}_i) = n$, and directed
edges $f_0, \dotsc, f_l$ from $v$ to $u$ with labels $\lambda(f_0) =
1$, $\lambda(\overline{f}_0) = n$, $\lambda(f_i) = n$,
$\lambda(\overline{f}_i) = n^2$ ($i \geq 1$). 
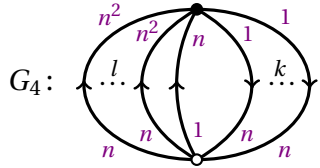
\[ G_4\colon \ \ 
\vcenter{\hbox{
\begin{tikzpicture}
\small
\draw[very thick] (2,2) ellipse (1.5 and 1);
\draw[very thick] (2,1) arc(-45:45:2.5 and 1.4142126);
\draw[very thick] (2,1) arc(225:135:2.5 and 1.4142126);
\draw[very thick] (2,1) arc(210:150:2);
\filldraw[fill=white,thick] (2,1) circle (.7mm);
\filldraw[fill=black,thick] (2,3) circle (.7mm);

\draw[very thick,->] (3.5,1.96) -- (3.5,1.95);
\draw[very thick,->] (2.735,1.96) -- (2.735,1.95);
\draw[very thick,->] (0.5,2.04) -- (0.5,2.05);
\draw[very thick,->] (1.265,2.04) -- (1.265,2.05);
\draw[very thick,->] (1.732,2.04) -- (1.732,2.05);
\draw[color=black] (0.9,1.97) node {$\dotsm$};
\draw[color=black] (3.15,1.97) node {$\dotsm$};

\scriptsize
\draw[color=black] (0.9, 2.01) node[anchor=south] {$l$}; 
\draw[color=black] (3.1, 2.01) node[anchor=south] {$k$}; 
\draw[violet] (0.845,2.933) node {$n^2$};
\draw[violet] (1.345,2.713) node {$n^2$};
\draw[violet] (2.02,2.58) node {$n$};
\draw[violet] (2.68,2.68) node {$1$};
\draw[violet] (3.18,2.9) node {$1$};
\draw[violet] (2.02,1.4) node {$1$};
\draw[violet] (2.66,1.32) node {$n$};
\draw[violet] (3.16,1.1) node {$n$};
\draw[violet] (1.32,1.32) node {$n$};
\draw[violet] (0.82,1.1) node {$n$};
\end{tikzpicture}
\hspace*{.85cm}
}}\]

\begin{theorem}\label{g4thm}
Suppose $n < k$ and $k \equiv 1 \mod n$. Then the lattices $G_1, G_4 <
\Aut(X_{k,kn})$ are not abstractly commensurable. 
\end{theorem}

\begin{proof}
Collapsing $f_0$ and then performing $2l+k-1$ slide moves yields 
  \begin{align*}
    G_4 \ &\cong \ \bigvee_l BS(n^2,n^2) \ \vee \ 
      \bigvee_{k} BS(1,n^2) \\
      &\cong \ BS(1,n^2) \ \vee \bigvee_{l+ k-1} BS(1,1).
  \end{align*}
The depth profile is $\{ \, n^{2i} \mid i \in \N \cup \{0\} \, \}$,
which is inequivalent to the depth profile of $G_1$. 
\end{proof}

\subsection{The remaining cases}\label{lastcase}
There are two remaining cases of $X_{k,kn}$ not covered by Theorems
\ref{mainthm}, \ref{g3thm}, and \ref{g4thm}. The first is when every
prime factor of $n$ is greater than $k$. The second is when $n$ is
prime, $n <k$, and $k \not\equiv 0,1 \mod n$ (for instance,
$X_{5,15}$). The constructions seen thus far all result in lattices
having equivalent depth profiles. Ad hoc arguments seem to indicate
that some of these examples (in the second case) are still
incommensurable, but a complete presentable proof is elusive so far.

\section{Cayley graphs}\label{cayleysec}

Given a finitely generated group $G$ let $S$ be a symmetric finite
generating set that does not contain $1$. The \emph{Cayley graph}
$\Cay(G,S)$ is a connected graph with edges labeled by elements of
$S$, defined as follows. The vertex set is $G$. For every $g\in G$ and
$s \in S$ there is a unique edge $e$ with $\partial_0(e) = g$ and
$\partial_1(e) = gs$. This edge is assigned the label $s$. Thus,
$\overline{e}$ has label $s^{-1}$. 

We say that $G_1$ and $G_2$ \emph{admit isomorphic Cayley graphs} if
there exist generating sets $S_1, S_2$ as above such that $\Cay(G_1,S_1)$
and $\Cay(G_2,S_2)$ are isomorphic as unlabeled graphs. Note that this
is not a transitive relation (see \cite{hainkescheele}). 

There are several interesting examples of groups that are quite
different from one another admitting isomorphic Cayley graphs. Most
involve torsion in an essential way, since finite groups of the same
cardinality always admit isomorphic Cayley graphs. If $A$ and $B$ are
two such finite groups, then $A \wr \Z$ and $B \wr \Z$ admit
isomorphic Cayley graphs, by \cite{dyubina}. (Such groups are never
finitely presented, however.) Similarly, if $G_A$ and $G_B$ are
extensions of $A$ and $B$ respectively by a group $Q$, then
$G_A$ and $G_B$ admit isomorphic Cayley graphs
(see \cite[proof of Corollary 1.13]{MMV} for instance). 

If one seeks incommensurable \emph{torsion-free} groups with
isomorphic Cayley graphs, then some well known examples can be found
among lattices in products of locally finite trees. This phenomenon
was first observed and discussed by Wise in \cite{wisethesis}; the
lattices of Burger and Mozes \cite{burgermozes} also provide
examples. In \cite{dergachevaklyachko} Dergacheva and Klyachko
constructed a pair of incommensurable torsion-free groups with
isomorphic Cayley graphs, via amalgams of Baumslag--Solitar
groups. Our lattices here provide new examples, which moreover are
also coherent. It is an open question whether there exist coherent
lattices in products of trees \cite[Problem 10.10]{wisecsc}. 

If a group acts cocompactly on a connected CW complex $X$, freely and
transitively on the vertices, then the $1$--skeleton $X^{(1)}$ is a
Cayley graph for $G$. This is the situation for the torsion-free
examples just mentioned. Unfortunately, our lattices $G_i$ don't act
in this way on $X_{k,kn}$. Instead we have the following result. 

\begin{proposition}\label{cayleyprop}
Suppose $G_1$ and $G_2$ act cocompactly on a connected graph $\Gamma$,
freely on the vertices, with a common vertex orbit. Then $G_1$ and
$G_2$ admit isomorphic Cayley graphs. 
\end{proposition}

Having a vertex orbit in common is important. For instance, two groups
may act on the same graph, with the same number of vertex orbits, and
that is not enough. In the group $G = \Z \times \Z/2$ with any Cayley
graph $\Gamma$, the subgroups $\Z \times \{0\}$ and $2\Z \times \Z/2$
both act on $\Gamma$ with two vertex orbits, but they do not admit
isomorphic Cayley graphs, by \cite{hainkescheele} (see also
\cite{loh}). 

\begin{proof}
The assumptions imply that $\Gamma$ is locally finite with bounded
valence. Also, there is a vertex $v_0 \in V(\Gamma)$ such that $G_1 v_0
= G_2 v_0$. Let $V_0$ denote this vertex orbit. Because the action of
$G_1$ is cocompact, there is a number $C$ such that every vertex of
$\Gamma$ has distance at most $C$ from $V_0$. 

Now let $P$ be the set of paths in $\Gamma$ of the form $\alpha \cdot
e \cdot \beta$ where 
\begin{itemize}
\item $e \in E(\Gamma)$
\item $\alpha$ is a shortest path in $\Gamma$ from $V_0$
  to $\partial_0 e$
\item $\beta$ is a shortest path in $\Gamma$ from $\partial_1 e$ to
  $V_0$. 
\end{itemize}
We build a new graph $\Delta'$ as follows: $V(\Delta')$ is the set
$V_0$ and $E(\Delta')$ is $P$. That is, each $p\in P$ has initial and
terminal endpoints $\partial_0 p$ and $\partial_1 p$ in $V_0$, which
defines an edge in $\Delta'$. (Note that $P$ is closed under the
involution $p \mapsto \overline{p}$.) Finally, define $\Delta$ from
$\Delta'$ by eliminating all loops and duplicate edges, if any exist. 

We claim that $\Delta$ is connected. Suppose $\gamma = (e_1, \dotsc,
e_n)$ is any path in $\Gamma$ with endpoints in $V_0$. For each $i$
let $\alpha_i$ be a shortest path from $V_0$ to $\partial_0 e_i$. Then
$(e_1 \cdot \overline{\alpha}_2) \cdot (\alpha_2 \cdot e_2 \cdot
\overline{\alpha}_3) \dotsm (\alpha_n \cdot e_n)$ is a concatenation
of paths in $P$ from $\partial_0 \gamma$ to $\partial_1 \gamma$. 

It is immediate that both $G_1$ and $G_2$ act on $\Delta$, freely and
transitively on the vertices. Next let 
\[S_i \  = \ \{\, g \in G_i \mid d_{\Delta}(v_0, gv_0) = 1\,\}\] 
for $i = 1,2$. These sets are finite because $d_{\Delta}(v_0, gv_0) =
1$ implies $d_{\Gamma}(v_0, gv_0) \leq 2C+1$ and $\Gamma$ is locally
finite. Finally, $\Cay(G_i, S_i) \cong \Delta$ since every path in $P$
has a unique translate under the action of $G_i$ with initial vertex
$v_0$. That is, every edge of $\Delta $ has a unique $G_i$--translate
which is an edge from $v_0$ to $s v_0$ for some $s \in S_i$. 
\end{proof}

\begin{corollary}\label{cayleycor}
The lattices $G_1, G_2, G_3, G_4 < \Aut(X_{k,kn})$ admit isomorphic Cayley
graphs. 
\end{corollary}

\begin{proof}
Recall that the labeled graphs defining $G_i$ have two vertices, black
and white. Lifting this coloring to the common Bass--Serre tree
$T_{k,kn}$ we get a bipartite vertex coloring. Now lift this vertex
coloring to the vertices of $X_{k,kn}$ using $\pi$. Each branching
line has only vertices of one color, and every strip has opposite
vertex colors on its two sides. Now note that the white
vertices and the black vertices are exactly the vertex orbits under
any of the group actions. Hence Proposition \ref{cayleyprop} applies,
with $\Gamma$ the $1$--skeleton of $X_{k,kn}$.  
\end{proof}

\section{Questions}\label{questions}

The main question is the following:

\begin{question}
  For which pairs $(m,n)$ does $\Aut(X_{m,n})$ contain incommensurable
  lattices? 
\end{question}
The main open cases are the two cases of $\Aut(X_{k,kn})$ from Section
\ref{lastcase} and the case when $\gcd(m,n) \not= 1$, $m \nmid n$, and
$n \nmid m$. 

\begin{question}
  What are the uniform lattices in $\Aut(X_{m,n})$ with torsion? 
  Are there uniform lattices that are not virtually torsion-free? 
\end{question}

% CHECK if the construction below actually acts on $X_{m,n}$ !! 
%
%\iffalse
For any $(A,\lambda)$ there is a straightforward construction of a
lattice in $\Aut(X_{(A,\lambda)})$ with $2$--torsion, via a 
graph of infinite dihedral groups. These examples contain torsion-free
subgroups of index $2$.
%If $\lambda$ is positive then this latter
%statement is easy to see; every torsion element is detected by
%the orientation character, so its kernel is torsion-free.
It would be
much more interesting to find lattices with higher-order torsion. 
%\fi

If $T$ is a locally finite tree then every uniform lattice in
$\Aut(T)$ is virtually torsion-free, by \cite{KPS}. On the other hand,
Hughes recently constructed a uniform lattice in $\Aut(T_1) \times
\Aut(T_2)$ that is not virtually torsion-free \cite{hughes}. The
reasoning used in \cite{hughes} seems unlikely to apply in our
setting, as it depends on the existence of a simple uniform lattice.

\begin{question}
  Does $\Aut(X_{m,n})$ contain non-uniform lattices? 
\end{question}

Carbone \cite{carbone} has shown that $\Aut(T)$ contains non-uniform lattices, for
any uniform locally finite tree $T$. R\'{e}my \cite{remy} has shown
that certain Kac-Moody groups over finite fields are irreducible
non-uniform lattices in $\Aut(T \times T)$. 
%I am not aware of any examples of irreducible non-uniform lattices in
%$\Aut(T_1 \times T_2)$ when $T_1 \not\cong T_2$. 

%(Specifically, for any $q$,  $\SL_2\left(\mathbb{F}_q[t,
%t^{-1}]\right)$ is a non-uniform lattice in $\Aut(T\times T)$, where
%$T$ is the Bruhat-Tits tree of $\SL_2\left(\mathbb{F}_q((t))\right)$.)

Note that $\Aut(X_{m,m}) \cong \Aut(T_{2m}) \times D_{\infty}$ for $m
> 1$, so it contains reducible non-uniform lattices, by
\cite{carbone}. Here $T_{2m}$ is the \emph{undirected} regular tree of
valence $2m$. 
%Hence we are really asking about the case $m\not= n$. 
Hence we are really asking about \emph{irreducible} non-uniform
lattices if $m=n$.

\begin{question}
  When do $\Aut(X_{m,n})$ and $\Aut(X_{m',n'})$ contain isomorphic
  lattices? 
\end{question}

We have seen that $\Aut(X_{k,k^2})$ and $\Aut(X_{k^2,k^4})$ contain
isomorphic lattices (namely, the finite-index subgroup from Figure
\ref{easycase}), even though their ``standard'' lattices $BS(k,k^2)$
and $BS(k^2,k^4)$ are not commensurable. This failure of rigidity for
lattices in $\Aut(X_{m,n})$ is intriguing. 

Finally, all of the above can be studied for the more general groups
$\Aut(X_{(A,\lambda)})$.

%\bibliographystyle{amsalpha}
%\bibliography{bib}

\providecommand{\bysame}{\leavevmode\hbox to3em{\hrulefill}\thinspace}
\providecommand{\MR}{\relax\ifhmode\unskip\space\fi MR }
% \MRhref is called by the amsart/book/proc definition of \MR.
\providecommand{\MRhref}[2]{%
  \href{http://www.ams.org/mathscinet-getitem?mr=#1}{#2}
}
\providecommand{\href}[2]{#2}

\end{document}